\documentclass[10pt]{amsart}
\usepackage{amsfonts}
\usepackage{hyperref}
\usepackage{amsmath,amssymb,amsthm}
\usepackage{enumerate,graphicx}
\usepackage[square,numbers]{natbib}
\usepackage{blkarray}

\newtheorem{theorem}{Theorem}[section]

\newtheorem{proposition}[theorem]{Proposition}
\newtheorem{corollary}[theorem]{Corollary}

\theoremstyle{definition}

\newtheorem{example}[theorem]{Example}

\newtheorem{problem}[theorem]{Problem}
\newtheorem{remark}[theorem]{Remark}

\begin{document}
\title[Entropy of $G$-SFTs]{Complexity of Shift Spaces on Semigroups}
\author{Jung-Chao Ban}
\thanks{*To whom correspondence should be addressed}
\address[Jung-Chao Ban]{Department of Applied Mathematics, National Dong Hwa
University, Hualien 970003, Taiwan, R.O.C.}
\email{jcban@gms.ndhu.edu.tw}
\author[Chih-Hung Chang]{Chih-Hung Chang*}
\address[Chih-Hung Chang]{Department of Applied Mathematics, National
University of Kaohsiung, Kaohsiung 81148, Taiwan, R.O.C.}
\email{chchang@nuk.edu.tw}
\author{Yu-Hsiung Huang }
\address[Yu-Hsiung Huang]{Department of Applied Mathematics, National Dong
Hwa University, Hualien 970003, Taiwan, R.O.C.}
\email{algobear@gmail.com}
\date{July 8, 2018}
\thanks{This work is partially supported by the Ministry of Science and
Technology, ROC (Contract No MOST 105-2115-M-259 -006 -MY2 and
105-2115-M-390 -001 -MY2) and National Center for Theoretical Sciences.}



\maketitle

\begin{abstract}
Let $G=\left\langle S|R_{A}\right\rangle $ be a semigroup with generating set $
S$ and equivalences $R_{A}$ among $S$ determined by a matrix $A$. This paper investigates the complexity of $G$-shift spaces by yielding the topological entropies. After revealing the existence of topological entropy of $G$-shift of finite type ($G$-SFT), the calculation of topological entropy of $G$-SFT is equivalent to solving a system of nonlinear recurrence equations. The complete characterization of topological entropies of $G$-SFTs on two symbols is addressed, which extends [Ban and Chang, arXiv:1803.03082] in which $G$ is a free semigroup.
\end{abstract}

% Title Page ----------------------------------------------------
\baselineskip=1.4 \baselineskip

\section{Introduction}

Let $\mathcal{A}$ be a finite alphabet and $G$ be a group. A \emph{%
configuration} is a function $f:G\rightarrow \mathcal{A}$ and a \emph{pattern%
} is a function from a finite subset of $G$ to $\mathcal{A}$. A subset $%
X\subseteq \mathcal{A}^{G}$ is called a \emph{shift space }if $X=X_{\mathcal{%
F}}$ which is a set of configurations which avoid patterns from some set $%
\mathcal{F}$ of patterns. If $\mathcal{F}$ is finite we call such a shift
space $\emph{shift}$ $\emph{of}$ $\emph{finite}$ $\emph{type}$ (SFT). Let $%
n\in \mathbb{N}$ and denote by $E_{n}$ the set of elements in $G$ whose
length are less than or equal to $n$. We define $\Gamma _{n}(X)$ the set of
all possible patterns of $X$ in $E_{n}$ and set $\gamma _{n}=\left\vert
\Gamma _{n}\right\vert $, i.e., the number of $\Gamma _{n}$. The \emph{%
topological entropy }(\emph{entropy }for short) of $X$ is defined as 
\begin{equation}
h(X)=\limsup\limits_{n\rightarrow \infty }\frac{\ln \gamma _{n}}{\left\vert
E_{n}\right\vert }\text{.}  \label{21}
\end{equation}%
It is known that the value (\ref{21}) measures the exponential growth rate
of the number of admissible patterns. From the dynamics viewpoint, it is a
measure of the randomness or complexity of a given physical or dynamical
system (cf. \cite%
{simpson2017symbolic,brin2002introduction,downarowicz2011entropy,ott2002chaos}%
). From the information theory viewpoint, it measures how much information
can be stored in the set of allowed sequences.

The value $h(X)$ (we write it simply $h$) exists for $G=\mathbb{Z}^{1}$
under the classical subadditive argument, and its entropy formula and
algebraic characterization is given by D. Lind \cite{Lind-ETaDS1984,LM-1995}%
; that is, the nonzero entropies of $\mathbb{Z}^{1}$-SFTs are exactly the
non-negative rational multiples of the logarithm of Perron numbers\footnote{%
A \emph{Perron number} is a real algebraic integer greater than $1$ and
greater than the modulus of its algebraic conjugates.}. Generally, if $G$ is 
\emph{amenable} ($\mathbb{Z}^{d}$ is amenable), $h$ exists due to the fact
that $G$ satisfies the F{\o}lner condition \cite%
{coornaert2015topological,kerr2016ergodic,ceccherini2010cellular}. The
characterization of the entropies for $G=\mathbb{Z}^{d}$, $d\geq 2$ is given
by Hochman-Meyerovitch \cite{HM-AoM2010}, namely, the entropies of such $G$%
-SFTs are the set of \emph{right recursively enumerable}\footnote{%
The infimum of a monotonic recursive sequence of rational numbers}\emph{\ }%
numbers. In \cite{BL-DCDS2005}, the authors found the recursive formula of $%
h $ in $\mathbb{Z}^{2}$-SFTs, some explicit value of $h$ can be computed
therein. Finally, the quantity (\ref{21}) can also be used to characterize
the chaotic spatial behavior for $\mathbb{Z}^{2}$ lattice dynamical systems
(cf. \cite{CS-SJAM1995, CM-ITCSIFTaA1995, CMV-IJBCASE1996, BCL+-JDE2009,
BCL-JDE2012}).

For $G=FS_{d}$, the free semigroup with generators $S=\{s_{1},\ldots
,s_{d}\} $ (it is not amenable), the existence of the limit (\ref{21}) is
due to the recent result of Petersen-Salama \cite{petersen2018tree}. Its
entropy formula for $d=2$ and $\left\vert \mathcal{A}\right\vert =2$ is
presented by Ban-Chang \cite{ban2018topological}. For $G=F_{2}$, the free
group with $2$ generators, Piantadosi \cite{piantadosi2008symbolic} finds an
approximation of $h\approx 0.909155$ for the $F_{2}$-\emph{golden mean shift}%
, i.e., the forbidden set $\mathcal{F}\subset \mathcal{A}\times S\times 
\mathcal{A}$ is defined by $\mathcal{F}=\{(2,s_{i},2):i=1,2\}$\footnote{%
This means that the consecutive $22$ is a forbidden pattern along the
generators $s_{1}$ and $s_{2}$.} for $\left\vert \mathcal{A}\right\vert =2$.
Either the existence of the limit of (\ref{21}) or to find the exact value
of $h$ for a $F_{d}$-SFT is an open problem. The specification properties or
the chaotic behavior of $FS_{d}$-SFTs can be found in \cite{ban2017tree(a),
ban2017mixing(b)}

A \emph{semigroup }is a set $G=\left\langle S|R\right\rangle $ together with
a binary operation which is closed and associative, where $R$ is a set of
equivalences which describe the relations among $S$. A \emph{monoid }is a
semigroup with an identity element $e$. Suppose $A\in \{0,1\}^{d\times d}$
is a binary matrix. A semigroup/monoid $G$ of the form $G=\left\langle
S|R_{A}\right\rangle $ means that $s_{i}s_{j}=s_{i}$ if and only if $%
A(i,j)=0 $. We note that $FS_{d}=\left\langle S|R_{A}\right\rangle $, where $%
S=\{s_{1},\ldots ,s_{d}\}$ and $A=\mathbf{E}_{d}$, the $d\times d$ matrix
with all entries being $1^{\prime }$s. The aim of this paper is to find the
entropy formula of $G$-SFTs. Although the discussion works for general cases, we focus on the case where $%
d=\left\vert \mathcal{A}\right\vert =2$ and $A=\left( 
\begin{array}{cc}
1 & 1 \\ 
1 & 0%
\end{array}%
\right) $ for the clarity and compactness of the idea.

In Section 2 we demonstrate that $\gamma _{n}$ solves some nonlinear
recurrence equation with respect to the lattice $G$ and the rules $%
T=(T_{1},T_{2})$ (Theorem \ref{Thm: 5}). Various types of recurrence
equations, namely, the \emph{zero entropy type }(type $\mathbf{Z}$,
Proposition \ref{Prop: 1}), \emph{equal growth type }(type $\mathbf{E}$,%
\textbf{\ }Theorem \ref{Thm: 2}), \emph{dominating type} (type $\mathbf{D}$,%
\textbf{\ }Theorem \ref{Thm: 3} and Theorem \ref{Thm: 7}) and \emph{%
oscillating types }(type $\mathbf{O}$, Proposition \ref{Prop: 3}) are
introduced, and the algorithms of the entropy computation for these types
are presented in Section 3. We give the complete characterization of $h$ of $%
G$-SFTs with $\left\vert \mathcal{A}\right\vert =2$ (Theorem \ref{Thm: 6})
in Section 4. That is, all the nonlinear recurrence equations of $G$-SFTs
with $\left\vert \mathcal{A}\right\vert =2$ are equal to $\mathbf{Z}\cup 
\mathbf{E}\cup \mathbf{D}\cup \mathbf{O}$. Some open problems related to
this topic are listed in Section 5.

\section{Preliminaries}

In this section, we give the notations and some known results of $G$-SFTs.
Let $\mathcal{A}$ be a finite set and $t\in \mathcal{A}^{G}$, for $g\in G$, $%
t_{g}=t(g)$ denotes the label attached to the vertex $g$ of the right \emph{%
Cayley graph}\footnote{%
The \emph{right Cayley graph} of $G$ with respect to $S$ is the directed
graph whose vertex set is $G$ and its set of arcs is given by $%
E=\{(g,gs):g\in G,s\in S\}$} of $G$. The \emph{full shift }$\mathcal{A}^{G}$
collects all configurations, and the \emph{shift map }$\sigma :G\times 
\mathcal{A}^{G}\rightarrow \mathcal{A}^{G}$ is defined as $(\sigma
_{g}t)_{g^{\prime }}=t_{gg^{\prime }}$ for $g,g^{\prime }\in G$. For $n\geq
0 $, let $E_{n}=\{g\in G:\left\vert g\right\vert \leq n\}$ denote the
initial $n$-subgraph of the Cayley graph. An $n$\emph{-block }is a function $%
\tau :E_{n}\rightarrow \mathcal{A}$. A configuration $t$ \emph{accepts }an $%
n $-block $\tau $ if there exists $g\in G$\emph{\ }such that\emph{\ }$%
t_{gg^{\prime }}=\tau _{g^{\prime }}$ for all $g^{\prime }\in E_{n}$;
otherwise, we call $\tau $ a \emph{forbidden block }of $t$ (or $t$ \emph{%
avoids }$\tau $). A $G$-shift space is a set $X\subseteq \mathcal{A}^{G}$ of
all configurations which avoid a set of forbidden blocks. For $n\in \mathbb{N%
}$ and $g\in G$, let $\Gamma _{n}^{[g]}(X)$ be the set of $n$-blocks of $X$
rooted at $g$, i.e., the support of each block of $\Gamma _{n}^{[g]}(X)$ is $%
gE_{n}$. Let $\gamma _{n}^{[g]}=\left\vert \Gamma _{n}^{[g]}(X)\right\vert $%
, the cardinality of $\Gamma _{n}^{[g]}(X)$, and denote $\gamma
_{n}(X)=\gamma _{n}^{[e]}(X)$. The \emph{topological degree }of $X$ is
defined as 
\begin{equation}
\deg (X)=\limsup\limits_{n\rightarrow \infty }\frac{\ln \ln \gamma _{n}(X)}{n%
}.  \label{22}
\end{equation}%
In \cite{ban2017coloring}, the authors show that the limit of (\ref{22})
exists for $G$-SFTs, where $G=\left\langle S|R_{A}\right\rangle $ and $%
A=\left( 
\begin{array}{cc}
1 & 1 \\ 
1 & 0%
\end{array}%
\right) $. The relation of $\deg (X)$ and $h$ is described as follows. For a 
$G$-SFT, $\gamma _{n}$ behaves approximately like $\lambda _{1}\lambda
_{2}^{\kappa ^{n}}$ for some $\kappa ,\lambda _{1},\lambda _{2}\in \mathbb{R}
$ while $\deg (X)=\ln \kappa $. The formulation $\lambda _{1}\lambda
_{2}^{\kappa ^{n}}$ reveals that we could use $\lambda _{2}$ colors (in
average) to \emph{fill up} the elements of $E_{n}$ in $G$. The $\deg (X)$
represents the logarithm of the degree $\kappa $ (in average) of $G$ in the
viewpoint of the graph theory. For examples, let $\mathcal{A}=\{1,2\}$ and $%
X $ be a $\mathbb{Z}^{1}$- full shift over $\mathcal{A}$, then $\gamma
_{n}=2^{n}$ while $\left\vert E_{n}\right\vert =n$ and $\deg (X)=0$. If $X$
is a $FS_{3}$-full shift with $\mathcal{A}$, then $\gamma _{n}=2^{3^{n}}$
while $\left\vert E_{n}\right\vert =3^{n}$ and $\deg (X)=\ln 3$. Thus the
formulation $\lambda _{1}\lambda _{2}^{\kappa ^{n}}$ can also be symbolized
as $\gamma _{n}\approx \left\vert \mathcal{A}\right\vert ^{\left\vert
E_{n}\right\vert }$. It can be easily checked that if $\deg (X)=\ln \rho
_{A} $, where $\rho _{A}$ denotes the the spectral radius of the matrix $A$,
then $h=\ln \lambda _{2}$. However, it is not always the case, we prove that 
$\{\kappa (X_{\mathcal{F}}):X_{\mathcal{F}}$ is a $FS_{d}$-SFT$\}=\{\xi ^{%
\frac{1}{p}}:\xi \in \mathcal{P}$, $p\geq 1\}$ \cite{ban2017tree} for $%
FS_{d} $-SFTs, where $\mathcal{P}$ is the set of Perron numbers. Suppose $%
G=\left\langle S|R_{A}\right\rangle $ has at least one \emph{right free
generator}\footnote{$s_{i}$ is a right (resp. left) free generator if and
only if $A(i,j)=1$ (resp. $A(j,i)=1$) for $1\leq j\leq d$.}, Ban-Chang-Huang 
\cite{ban2018topological1} provide the necessary and sufficient conditions
for $\kappa =\rho _{A}$. Roughly speaking, the more information of the tuple 
$(\kappa ,\lambda _{1},\lambda _{2})\in \mathbb{R}^{3}$ we know, the more
explicit value of $\gamma _{n}$ we obtain. A natural question arises: \emph{%
how to compute the value }$\lambda _{2}$\emph{?} If $G=FS_{2}$, i.e., $A=%
\mathbf{E}_{2}$, the authors give the complete characterization of the
entropies for $\left\vert \mathcal{A}\right\vert =2$ \cite%
{ban2017characterization}. This study is to extend the previous work to $%
A=\left( 
\begin{array}{cc}
1 & 1 \\ 
1 & 0%
\end{array}%
\right) $.

\subsection{The existence of the entropy}

From now on, we assume that $G=\left\langle S|R_{A}\right\rangle $, where $%
S=\{s_{1},s_{2}\}$, $A=$ $\left( 
\begin{array}{cc}
1 & 1 \\ 
1 & 0%
\end{array}%
\right) $ and $\mathcal{A}=\{1,2\}$. Theorem \ref{Thm: 4} below shows that
the limit (\ref{21}) exists for $G$-shifts. The proof is based on the
concept of the proof for $FS_{2}$-shifts \cite{petersen2018tree}. Since the
proof for $G$-shifts is not straightforward, we include the proof for the
convenience of the reader.

\begin{theorem}
\label{Thm: 4}Let $X$ be a $G$-shift. The entropy (\ref{21}) exists.
\end{theorem}

\begin{proof}
Let $L_{n}=\{g\in G:\left\vert g\right\vert =n\}$ and $l_{n}=\left\vert
L_{n}\right\vert $ for $n\geq 0$. Set 
\begin{equation*}
\underline{h}=\liminf\limits_{n\rightarrow \infty }\frac{\ln \gamma _{n}}{%
\left\vert E_{n}\right\vert }=\liminf\limits_{n\rightarrow \infty }\frac{\ln
\gamma _{n}}{l_{0}+l_{1}+\cdots +l_{n}}.
\end{equation*}%
Since $l_{0}=1$, for $\varepsilon >0$, there exists a large $m\in \mathbb{N}$
such that 
\begin{equation*}
\frac{\ln \gamma _{m}}{l_{1}+\cdots +l_{m}}<\underline{h}+\varepsilon \text{.%
}
\end{equation*}%
Write $n=pm+q$ where $0\leq q\leq m-1$, then we have 
\begin{eqnarray*}
\gamma _{pm+q} &=&\gamma _{\lbrack (p-1)m+q]+m}\leq \gamma _{(p-1)m+q}\gamma
_{m}^{l_{(p-1)m+q}} \\
&\leq &(\gamma _{(p-2)m+q}\gamma _{m}^{l_{(p-2)m+q}})\gamma
_{m}^{l_{(p-1)m+q}} \\
&&\vdots \\
&\leq &\gamma _{q}\gamma _{m}^{l_{q}+l_{m+q}+\cdots +l_{(p-1)m+q}}\text{.}
\end{eqnarray*}%
On the other hand, we have 
\begin{equation*}
\left( 
\begin{array}{c}
l_{n+1} \\ 
l_{n}%
\end{array}%
\right) =A\left( 
\begin{array}{c}
l_{n} \\ 
l_{n-1}%
\end{array}%
\right) \text{.}
\end{equation*}%
It can be easily checked that $l_{m+n}=l_{m}l_{n+1}+l_{m-1}l_{n}$. Thus we
have $l_{m+n}>l_{m}l_{n+1}>l_{m}l_{n}$ and 
\begin{eqnarray*}
l_{q+1}+l_{q+2}+\cdots +l_{q+pm} &>&l_{q}(l_{1}+\cdots
+l_{m})+l_{q+m}(l_{1}+\cdots +l_{m}) \\
&&+\cdots +l_{q+(p-1)m}(l_{1}+\cdots +l_{m}) \\
&=&(l_{1}+\cdots +l_{m})(l_{q}+l_{q+m}+\cdots +l_{q+(p-1)m}).
\end{eqnarray*}%
Therefore, for $n$ large enough we have%
\begin{eqnarray*}
\underline{h} &\leq &\frac{\ln \gamma _{n}}{\left\vert E_{n}\right\vert }=%
\frac{\ln \gamma _{pm+q}}{\left\vert E_{pm+q}\right\vert } \\
&\leq &\frac{\ln \gamma _{q}}{\left\vert E_{pm+q}\right\vert }+\frac{%
l_{q}+l_{q+m}+\cdots +l_{q+(p-1)m}}{l_{1}+l_{2}+\cdots +l_{q+pm}}\ln \gamma
_{m} \\
&\leq &\frac{\ln \gamma _{q}}{\left\vert E_{pm+q}\right\vert }+\frac{%
l_{q}+l_{q+m}+\cdots +l_{q+(p-1)m}}{l_{q+1}+\cdots +l_{q+pm}}\ln \gamma _{m}
\\
&\leq &\frac{\ln \gamma _{q}}{\left\vert E_{pm+q}\right\vert }+\frac{\ln
\gamma _{m}}{l_{1}+\cdots +l_{m}}<\underline{h}+2\varepsilon \text{.}
\end{eqnarray*}%
Thus we conclude that (\ref{21}) exists and equals $\underline{h}$. This
completes the proof.
\end{proof}

\subsection{ Nonlinear recurrence equations}

Let $\mathcal{F}\subseteq \mathcal{A}\times S\times \mathcal{A}$ be a set of
forbidden blocks and $X=X_{\mathcal{F}}$ be the corresponding $G$-SFT. The
associated \emph{adjacency matrices} $T_{1}=(t_{ij}^{1})_{i,j=1}^{2}$ and $%
T_{2}=(t_{ij}^{2})_{i,j=1}^{2}$ of $\mathcal{F}$ are defined as follows. For 
$i=1,2$, $T_{i}(a,b)=0$ if $(a,s_{i},b)\in \mathcal{F}$ and $T_{i}(a,b)=1$,
otherwise. Let $T=(T_{1},T_{2})$, the $G$\emph{-vertex shift} $X_{T}$ is
defined. 
\begin{equation}
X_{T}=\{t\in \mathcal{A}^{G}:T_{i}(t_{g},t_{gs_{i}})=1\text{ }\forall
g,gs_{i}\in G\}\text{.}  \label{7}
\end{equation}%
It is obvious that $X_{T}$ is equal to $X$ and thus $h(X)=h(X_{T})$.
Throughout the paper, we assume that $T_{i}$ has no zero rows for $i=1,2.$

Fix $i\in \mathcal{A}$, we set $\Gamma _{i,n}^{[g]}$ the set which consists
of all $n$-blocks $\tau $ in $\Gamma _{n}^{[g]}$ such that $\tau _{g}=i$ and 
$\gamma _{i,n}^{[g]}=\left\vert \Gamma _{i,n}^{[g]}\right\vert $, the
cardinality of $\Gamma _{i,n}^{[g]}$. For $g\in G$, we define $\mathbf{F}%
_{g}=\{g^{\prime }\in G:gg^{\prime }\in G$ and $\left\vert gg^{\prime
}\right\vert =\left\vert g\right\vert +\left\vert g^{\prime }\right\vert \}$%
. If $g=g_{1}g_{2}\cdots g_{n}\in G$, it is easily seen that $\mathbf{F}%
_{g}=G$ if $g_{n}=s_{1}$ and $\mathbf{F}_{g}=\mathbf{F}_{s_{2}}$ if $%
g_{n}=s_{2}$, note that $\mathbf{F}_{s_{2}}\neq G$. Thus 
\begin{eqnarray}
\gamma _{i,n}^{[s_{1}]}
&=&\sum_{j_{1},j_{2}=1}^{2}t_{ij_{1}}^{1}t_{ij_{2}}^{2}\gamma
_{j_{1},n-1}^{[s_{1}s_{1}]}\gamma
_{j_{2},n-1}^{[s_{1}s_{2}]}=%
\sum_{j_{1},j_{2}=1}^{2}t_{ij_{1}}^{1}t_{ij_{2}}^{2}\gamma
_{j_{1},n-1}^{[s_{1}]}\gamma _{j_{2},n-1}^{[s_{2}]},  \label{15} \\
\gamma _{i,n}^{[s_{2}]} &=&\sum_{j=1}^{2}t_{ij}^{1}\gamma
_{j,n-1}^{[s_{2}s_{1}]}=\sum_{j=1}^{2}t_{ij}^{1}\gamma
_{j,n-1}^{[s_{1}]},1\leq i\leq 2\text{ and }n\geq 3.  \label{16}
\end{eqnarray}%
The first equality in (\ref{15}) means that the $\Gamma _{i,n}^{[s_{1}]}$ is
generated by $\Gamma _{j_{1},n-1}^{[s_{1}s_{1}]}$ (resp. $\Gamma
_{j_{2},n-1}^{[s_{1}s_{2}]}$) according to the rule of $t_{ij_{1}}^{1}$
(resp. $t_{ij_{2}}^{2}$) of $T_{1}$ (resp. $T_{2}$). Since $\Gamma
_{j_{1},n-1}^{[s_{1}s_{1}]}$ and $\Gamma _{j_{2},n-1}^{[s_{1}s_{2}]}$ are in
different branches, the summation $%
\sum_{j_{1},j_{2}=1}^{2}t_{ij_{1}}^{1}t_{ij_{2}}^{2}\gamma
_{j_{1},n-1}^{[s_{1}s_{1}]}\gamma _{j_{2},n-1}^{[s_{1}s_{2}]}$ demonstrates
the cardinality of $\Gamma _{i,n}^{[s_{1}]}$. The second equality in (\ref%
{15}) comes from the fact that $\mathbf{F}_{s_{1}s_{1}}=\mathbf{F}_{s_{1}}$
and $\mathbf{F}_{s_{1}s_{2}}=\mathbf{F}_{s_{2}}$. Similar argument derives (%
\ref{16}), the only difference is that we do not have the item $\gamma
_{j,n-1}^{[s_{2}s_{2}]}$ in the summation since $A(2,2)=0$, i.e., $%
s_{2}s_{2}=s_{2}$. Since the equations (\ref{15}) and (\ref{16}) are the
nonlinear recurrence equations which describe the numbers $\gamma
_{i,n}^{[s_{1}]}$ and $\gamma _{i,n}^{[s_{2}]}$ for $i=1,2$, we continue to
write $\{\gamma _{i,n}^{[s_{1}]},\gamma _{i,n}^{[s_{2}]}\}_{i=1}^{2}$ to
represent nonlinear recurrence equations (\ref{15}) and (\ref{16}) by abuse
of notation.

The nonlinear recurrence equation $\{\gamma _{i,n}^{[s_{1}]},\gamma
_{i,n}^{[s_{2}]}\}_{i=1}^{2}$ can be described in an efficient way. Let $%
\mathcal{M}_{m}$ be the collection of $m\times m$ binary matrices. Let $%
v=(v_{i})_{i=1}^{n}$ and $w=(w_{i})_{i=1}^{n}$ be two vectors over $\mathbb{R%
}$. Denote by $\otimes $ the \emph{dyadic product }of $v$ and $w$; that is, 
\begin{equation*}
v\otimes w=(v_{1}w_{1},v_{1}w_{2},\ldots ,v_{1}w_{n},\ldots
,v_{n}w_{1},v_{n}w_{2},\ldots ,v_{n}w_{n})\text{.}
\end{equation*}

Let $M\in \mathcal{M}_{m}$ and $M^{(i)}$ denote the $i$th row of $M$. Define 
\begin{equation}
\alpha _{i}=T_{1}^{(i)}\otimes T_{2}^{(i)}\text{ for }i=1,2  \label{8}
\end{equation}%
and 
\begin{equation}
\Theta _{n}^{[s_{j}]}=\left\{ 
\begin{array}{ccc}
(\gamma _{1,n}^{[s_{1}]},\gamma _{2,n}^{[s_{1}]})\otimes (\gamma
_{1,n}^{[s_{2}]},\gamma _{2,n}^{[s_{2}]}), &  & \text{if }j=1; \\ 
(\gamma _{1,n}^{[s_{1}]},\gamma _{2,n}^{[s_{1}]})\otimes (1,1), &  & \text{%
if }j=2.%
\end{array}%
\right.  \label{11}
\end{equation}

Let \textquotedblleft $\cdot $\textquotedblright\ be the usual inner
product, we define 
\begin{equation}
\widehat{\gamma }_{i,n}^{[s_{j}]}=\alpha _{i}\cdot \Theta _{n-1}^{[s_{j}]}%
\text{ for }1\leq i,j\leq 2.  \label{24}
\end{equation}

\begin{theorem}
\label{Thm: 5} Let $T=(T_{1},T_{2})\in \mathcal{M}_{2}\times \mathcal{M}_{2}$%
, the formula (\ref{15}) and (\ref{16}) can be reformulated as follows. For $%
i=1,2$

\begin{enumerate}
\item $\gamma _{i,n}^{[s_{1}]}=\widehat{\gamma }_{i,n}^{[s_{1}]}$,

\item $\gamma _{i,n}^{[s_{2}]}$ is derived from $\widehat{\gamma }%
_{i,n}^{[s_{2}]}$ by letting all coefficients of the items $\gamma
_{1,n-1}^{[s_{1}]}$ and $\gamma _{2,n-1}^{[s_{1}]}$ in $\widehat{\gamma }%
_{i,n}^{[s_{2}]}$ to be $1$.
\end{enumerate}
\end{theorem}

\begin{proof}
It follows from (\ref{15}) 
\begin{eqnarray*}
\gamma _{i,n}^{[s_{1}]}
&=&\sum_{j_{1},j_{2}=1}^{2}t_{ij_{1}}^{1}t_{ij_{2}}^{2}\gamma
_{j_{1},n-1}^{[s_{1}s_{1}]}\gamma _{j_{2},n-1}^{[s_{1}s_{2}]} \\
&=&\left( \sum_{j_{1}=1}^{2}t_{ij_{1}}^{1}\gamma
_{j_{1},n-1}^{[s_{1}s_{1}]}\right) \left(
\sum_{j_{2}=1}^{2}t_{ij_{2}}^{2}\gamma _{j_{2},n-1}^{[s_{1}s_{2}]}\right)
\end{eqnarray*}%
Since $\gamma _{j_{1},n-1}^{[s_{1}s_{1}]}=\gamma _{j_{1},n-1}^{[s_{1}]}$ and 
$\gamma _{j_{1},n-1}^{[s_{1}s_{2}]}=\gamma _{j_{1},n-1}^{[s_{2}]}$, we
conclude that 
\begin{eqnarray*}
\gamma _{i,n}^{[s_{1}]} &=&\left[ \left( t_{i1}^{1},t_{i2}^{1}\right)
\otimes \left( t_{i1}^{2},t_{i2}^{2}\right) \right] \cdot \left[ \left(
\gamma _{1,n-1}^{[s_{1}]},\gamma _{2,n-1}^{[s_{1}]}\right) \otimes \left(
\gamma _{1,n-1}^{[s_{2}]},\gamma _{2,n-1}^{[s_{2}]}\right) \right] \\
&=&\alpha _{i}\cdot \Theta _{n-1}^{[s_{j}]}=\widehat{\gamma }_{i,n}^{[s_{1}]}%
\text{.}
\end{eqnarray*}
Thus $\gamma _{i,n}^{[s_{1}]}=\widehat{\gamma }_{i,n}^{[s_{1}]}$. On the
other hand, it follows from (\ref{24}) we have 
\begin{eqnarray}
\widehat{\gamma }_{i,n}^{[s_{2}]} &=&\alpha _{i}\cdot \Theta _{n-1}^{[s_{2}]}
\notag \\
&=&\left(
t_{i1}^{1}t_{i1}^{2},t_{i1}^{1}t_{i2}^{2},t_{i2}^{1}t_{i1}^{2},t_{i2}^{1}t_{i2}^{2}\right) \cdot 
\left[ (\gamma _{1,n-1}^{[s_{1}]},\gamma _{2,n-1}^{[s_{1}]})\otimes (1,1)%
\right]  \notag \\
&=&\left(
t_{i1}^{1}t_{i1}^{2},t_{i1}^{1}t_{i2}^{2},t_{i2}^{1}t_{i1}^{2},t_{i2}^{1}t_{i2}^{2}\right) \cdot \left( \gamma _{1,n-1}^{[s_{1}]},\gamma _{1,n-1}^{[s_{1}]},\gamma _{2,n-1}^{[s_{1}]},\gamma _{2,n-1}^{[s_{1}]}\right)
\notag \\
&=&t_{i1}^{1}t_{i1}^{2}\gamma _{1,n-1}^{[s_{1}]}+t_{i1}^{1}t_{i2}^{2}\gamma
_{1,n-1}^{[s_{1}]}+t_{i2}^{1}t_{i1}^{2}\gamma
_{2,n-1}^{[s_{1}]}+t_{i2}^{1}t_{i2}^{2}\gamma _{2,n-1}^{[s_{1}]}  \label{25}
\end{eqnarray}

Suppose that there is no restriction on the node $s_{2}$ in $G$, the same
reasoning as $\gamma _{i,n}^{[s_{1}]}$ applied to $\gamma _{i,n}^{[s_{2}]}$
implies 
\begin{eqnarray}
\gamma _{i,n}^{[s_{2}]} &=&t_{i1}^{1}t_{i1}^{2}\gamma
_{1,n-1}^{[s_{1}]}\gamma _{1,n-1}^{[s_{2}]}+t_{i1}^{1}t_{i2}^{2}\gamma
_{1,n-1}^{[s_{1}]}\gamma _{2,n-1}^{[s_{2}]}  \label{27} \\
&&+t_{i2}^{1}t_{i1}^{2}\gamma _{2,n-1}^{[s_{1}]}\gamma
_{1,n-1}^{[s_{2}]}+t_{i2}^{1}t_{i2}^{2}\gamma _{2,n-1}^{[s_{1}]}\gamma
_{2,n-1}^{[s_{2}]}.  \notag
\end{eqnarray}%
Since $s_{2}s_{2}=s_{2}$, the formula (\ref{25}) is derived from (\ref{27})
by letting $\gamma _{i,n-1}^{[s_{2}]}=1$ for $i=1,2$. However, compare to
the formula (\ref{16}) $\gamma
_{i,n}^{[s_{2}]}=\sum_{j=1}^{2}t_{ij}^{1}\gamma _{j,n-1}^{[s_{1}]}$, the
formula (\ref{25}) counts $\gamma _{i,n}^{[s_{2}]}$ repeatedly, e.g., if $%
t_{i1}^{1}t_{i1}^{2}=t_{i1}^{1}t_{i2}^{2}=1$, then (\ref{25}) counts the
item $\gamma _{1,n-1}^{[s_{1}]}$ twice. Thus $\gamma _{i,n}^{[s_{2}]}$ can
be derived from $\widehat{\gamma }_{i,n}^{[s_{2}]}$ by letting all
coefficients of the terms $\gamma _{1,n-1}^{[s_{1}]}$ and $\gamma
_{2,n-1}^{[s_{1}]}$ in $\widehat{\gamma }_{i,n}^{[s_{j}]}$ to be $1$. This
completes the proof.
\end{proof}

It is worth noting that the intrinsic meaning of (\ref{24}) is that the
effect of $T$ (\emph{rules}) comes from the factor $\alpha _{i}$ (since $%
\alpha _{i}=\alpha _{i}(T)$) and the effect of $G$ (\emph{lattice}) comes
from the factor $\Theta _{n-1}^{[s_{j}]}$.

\begin{example}
\label{Exam: 2}Let $T=(T_{1},T_{2})$, where $T_{i}=\left( 
\begin{array}{cc}
1 & 1 \\ 
1 & 0%
\end{array}%
\right) $ for $i=1,2$. Then we have $\alpha _{1}=(1,1,1,1)$ and $\alpha
_{2}=(1,0,0,0)$. Apply Theorem \ref{Thm: 5} we have $\gamma
_{1,n}^{[s_{1}]}=\sum_{i,j=1}^{2}\gamma _{i,n-1}^{[s_{1}]}\gamma
_{j,n-1}^{[s_{2}]}$ and $\gamma _{2,n}^{[s_{1}]}=\gamma
_{1,n-1}^{[s_{1}]}\gamma _{1,n-1}^{[s_{2}]}$. On the other hand, it follows
from $\widehat{\gamma }_{1,n}^{[s_{2}]}=2\gamma _{1,n-1}^{[s_{1}]}+2\gamma
_{2,n-1}^{[s_{1}]}$ and $\widehat{\gamma }_{2,n}^{[s_{2}]}=\gamma
_{1,n-1}^{[s_{1}]}$ that we have $\gamma _{1,n}^{[s_{2}]}=\gamma
_{1,n-1}^{[s_{1}]}+\gamma _{2,n-1}^{[s_{1}]}$ and $\gamma
_{2,n}^{[s_{2}]}=\gamma _{1,n-1}^{[s_{1}]}$.
\end{example}

Since $T=(T_{1},T_{2})\in \mathcal{M}_{2}\times \mathcal{M}_{2}$, there are
only finite possibilities of $T_{i}^{(1)}$ and $T_{i}^{(2)}$, namely, $%
(1,1),(1,0)$ and $(0,1)$. Hence we have only finite choices of $\alpha _{i}$
(recall (\ref{8})) for $i=1,2$ as follows. 
\begin{eqnarray*}
v_{1} &=&(1,1,1,1),v_{2}=(1,0,1,0),v_{3}=(0,1,0,1),v_{4}=(1,1,0,0), \\
v_{5} &=&(0,0,1,1),v_{6}=(1,0,0,0),v_{7}=(0,1,0,0),v_{8}=(0,0,1,0), \\
v_{9} &=&(0,0,0,1).
\end{eqnarray*}%
For the convenience of the discussion, we define $F_{kl}=\{\gamma
_{i,n}^{[s_{1}]},\gamma _{i,n}^{[s_{2}]}\}_{i=1}^{2}$ the nonlinear
recurrence equation of Theorem \ref{Thm: 5} by choosing $\left( \alpha
_{1},\alpha _{2}\right) =\left( v_{k},v_{l}\right) $, and $h=h_{kl}$ if the
corresponding recurrence equation is $F_{kl}$.

\begin{remark}
\label{Rmk: 1}

\begin{enumerate}
\item Given $\left( \alpha _{1},\alpha _{2}\right) =\left(
v_{k},v_{l}\right) $, the pair $(T_{1},T_{2})$ is also uniquely determined.
For instance, if $\left( \alpha _{1},\alpha _{2}\right) =\left(
v_{2},v_{3}\right) $, then $T_{1}^{(1)}=(1,1)$, $T_{2}^{(1)}=(1,0)$, $%
T_{1}^{(2)}=(1,1)$ and $T_{1}^{(2)}=(0,1)$. Thus, one can reconstruct $%
T=\left( T_{1},T_{2}\right) $ as 
\begin{equation}
T_{1}=\left( 
\begin{array}{cc}
1 & 1 \\ 
1 & 1%
\end{array}%
\right) \text{ and }T_{1}=\left( 
\begin{array}{cc}
1 & 0 \\ 
0 & 1%
\end{array}%
\right) \text{.}  \label{37}
\end{equation}

\item Note that $\mathbf{F}_{g}=G$ if $g_{n}=s_{1}$. The entropy (\ref{21})
can also be represented as 
\begin{equation}
h=\lim_{n\rightarrow \infty }\frac{\ln \gamma _{n}}{\left\vert
E_{n}\right\vert }=\lim_{n\rightarrow \infty }\frac{\ln \left( \gamma
_{1,n}^{[s_{1}]}+\gamma _{2,n}^{[s_{1}]}\right) }{\left\vert
E_{n}\right\vert }\text{,}  \label{17}
\end{equation}%
where the existence of the limit is due to Theorem \ref{Thm: 4}.
\end{enumerate}
\end{remark}

\subsection{Equivalence of the recurrence equations}

Given two nonlinear recurrence equations $F_{kl}$ and $F_{pq}$, we say that $%
F_{kl}$ is \emph{equivalent} to $F_{pq}$ (write $F_{kl}\simeq F_{pq}$) if $%
F_{kl}$ is equal to $F_{pq}$ by interchanging items $\gamma _{1,n}^{[s_{1}]}$
with $\gamma _{2,n}^{[s_{1}]}$ and $\gamma _{1,n}^{[s_{2}]}$ with $\gamma
_{2,n}^{[s_{2}]}$. It follows from (\ref{17}) that the entropies of two $G$%
-SFTs are equal if their corresponding nonlinear recurrence equations are
equivalent.

\begin{example}
$F_{48}\simeq F_{75}$.
\end{example}

\begin{proof}
It follows from Theorem \ref{Thm: 5} we obtain 
\begin{equation*}
F_{48}=\left\{ 
\begin{array}{c}
\gamma _{1,n}^{[s_{1}]}=\gamma _{1,n-1}^{[s_{1}]}\gamma
_{1,n-1}^{[s_{2}]}+\gamma _{1,n-1}^{[s_{1}]}\gamma _{2,n-1}^{[s_{2}]}, \\ 
\gamma _{2,n}^{[s_{1}]}=\gamma _{2,n-1}^{[s_{1}]}\gamma _{1,n-1}^{[s_{2}]},
\\ 
\gamma _{1,n}^{[s_{2}]}=\gamma _{1,n-1}^{[s_{1}]}, \\ 
\gamma _{2,n}^{[s_{2}]}=\gamma _{2,n-1}^{[s_{1}]}, \\ 
\gamma _{1,1}^{[s_{1}]}=2,\gamma _{2,1}^{[s_{1}]}=1,\gamma
_{1,1}^{[s_{2}]}=1,\gamma _{2,1}^{[s_{2}]}=1.%
\end{array}%
\right.
\end{equation*}%
If we interchange $\gamma _{1,n}^{[s_{1}]}$ (resp. $\gamma _{1,n}^{[s_{2}]}$%
) with $\gamma _{2,n}^{[s_{1}]}$ (resp. $\gamma _{2,n}^{[s_{2}]}$) we have%
\begin{equation*}
\left\{ 
\begin{array}{c}
\gamma _{2,n}^{[s_{1}]}=\gamma _{2,n-1}^{[s_{1}]}\gamma
_{2,n-1}^{[s_{2}]}+\gamma _{2,n-1}^{[s_{1}]}\gamma _{1,n-1}^{[s_{2}]}, \\ 
\gamma _{1,n}^{[s_{1}]}=\gamma _{1,n-1}^{[s_{1}]}\gamma _{2,n-1}^{[s_{2}]},
\\ 
\gamma _{2,n}^{[s_{2}]}=\gamma _{2,n-1}^{[s_{1}]}, \\ 
\gamma _{1,n}^{[s_{2}]}=\gamma _{1,n-1}^{[s_{1}]}, \\ 
\gamma _{1,1}^{[s_{1}]}=1,\gamma _{2,1}^{[s_{1}]}=2,\gamma
_{1,1}^{[s_{2}]}=1,\gamma _{2,1}^{[s_{2}]}=1.%
\end{array}%
\right.
\end{equation*}%
One can check that it is indeed $F_{75}$. This completes the proof.
\end{proof}

\section{Formula and estimate of entropy}

In what follows, $\lambda $ stands for the \emph{spectral radius} of $A$,
i.e., $\lambda =\frac{1+\sqrt{5}}{2}$ and $\bar{\lambda}$ is its conjugate.
We provide various types of nonlinear recurrence equations in which the
formula (or estimate) of $h$ are presented in this section. By abuse of
notation we also denote by $\left\vert v\right\vert
=\sum_{i=1}^{n}\left\vert v^{(i)}\right\vert $ the\emph{\ norm} of $v\in 
\mathbb{R}^{n}$ and $v^{(i)}$ the $i$th coordinate of $v$. It can be easily
checked that $\left\vert E_{n}\right\vert =\left( \sum_{i=0}^{n}A^{i}\mathbf{%
1}\right) ^{(1)}$, where $\mathbf{1}=(1,1)^{\prime }$ and $v^{\prime }$
denotes the \emph{transpose} of $v$.

\subsection{Zero entropy type}

Proposition \ref{Prop: 1} below indicates that $h_{kl}=0$ if $F_{kl}$
satisfies $\left\vert v_{k}\right\vert =\left\vert v_{l}\right\vert =1$,
e.g., $k,l=6,7,8,9$. We call such $F_{kl}$ \emph{zero entropy type }(write
type $\mathbf{Z}$).

\begin{proposition}
\label{Prop: 1}Let $T=(T_{1},T_{2})\in \mathcal{M}_{2}\times \mathcal{M}_{2}$%
, and $\alpha _{i}=\alpha _{i}(T)$ be defined as above for $i=1,2$. If $%
\left\vert \alpha _{1}\right\vert =\left\vert \alpha _{2}\right\vert =1$,
then $h=0$.
\end{proposition}

\begin{proof}
For simplicity we only prove the case $F_{67}$, the other cases can be
treated similarly. Indeed, $F_{67}$ is of the following form. 
\begin{equation}
\left\{ 
\begin{array}{c}
\gamma _{1,n}^{[s_{1}]}=\gamma _{1,n-1}^{[s_{1}]}\gamma _{1,n-1}^{[s_{2}]},
\\ 
\gamma _{2,n}^{[s_{1}]}=\gamma _{1,n-1}^{[s_{1}]}\gamma _{2,n-1}^{[s_{2}]},
\\ 
\gamma _{1,n}^{[s_{2}]}=\gamma _{1,n-1}^{[s_{1}]}, \\ 
\gamma _{2,n}^{[s_{2}]}=\gamma _{1,n-1}^{[s_{1}]}, \\ 
\gamma _{1,1}^{[s_{1}]}=\gamma _{2,1}^{[s_{1}]}=\gamma
_{1,1}^{[s_{2}]}=\gamma _{2,1}^{[s_{2}]}=1%
\end{array}%
\right.  \label{36}
\end{equation}%
Note that $\gamma _{i,1}^{[s_{j}]}=1$, and if we assume that $\gamma
_{i,k-1}^{[s_{j}]}=1$ for $1\leq i,j\leq 2$, (\ref{36}) infers that $\gamma
_{i,k}^{[s_{j}]}=1$. Thus $\gamma _{i,n}^{[s_{j}]}=1$ for all $1\leq n$ and $%
1\leq i,j\leq 2$ by induction. This shows that $h_{67}=0$. This completes
the proof.
\end{proof}

\subsection{Equal growth type}

Let $T=(T_{1},T_{2})$ and $F=\{\gamma _{i,n}^{[s_{1}]},\gamma
_{i,n}^{[s_{2}]}\}_{i=1}^{2}$ be its nonlinear recurrence equations, we say
that $F$ is of the \emph{equal growth type} ($F\in \mathbf{E}$) if $%
\left\vert \alpha _{1}\right\vert =\left\vert \alpha _{2}\right\vert $.
Denote by $k_{i,j}$ the number of different items of $\gamma
_{i,n}^{[s_{j}]} $ for $1\leq i,j\leq 2$. If $\left\vert \alpha
_{1}\right\vert =\left\vert \alpha _{2}\right\vert $, it can be checked that 
$k_{1,1}=k_{2,1}=\alpha $, but $k_{1,2}$ may not equal to $k_{2,2}$ in
general.

\begin{theorem}
\label{Thm: 2}Let $T=(T_{1},T_{2})$ and the corresponding $\alpha
_{1},\alpha _{2}$ satisfy $\left\vert \alpha _{1}\right\vert =\left\vert
\alpha _{2}\right\vert =\alpha \in \mathbb{N}$. If $k_{1,2}=k_{2,2}=:\beta $%
, then 
\begin{equation*}
h=\left( \frac{1-\bar{\lambda}\frac{\ln \beta }{\ln \alpha }}{\lambda ^{2}}%
\right) \ln \alpha \text{.}
\end{equation*}%
Furthermore, if $k_{1,2}=k_{2,2}=\alpha $, then $h=\frac{\ln \alpha }{%
\lambda }$.
\end{theorem}

\begin{proof}
\textbf{1} First we claim that $\gamma _{1,n}^{[s_{j}]}=\gamma
_{2,n}^{[s_{j}]}$ for $1\leq j\leq 2$ and we prove it by induction. Note
that $\gamma _{1,0}^{[s_{j}]}=\gamma _{2,0}^{[s_{j}]}=1$ and assume that $%
\gamma _{1,n}^{[s_{j}]}=\gamma _{2,n}^{[s_{j}]}$ for $1\leq j\leq 2$.
Theorem \ref{Thm: 5} is applied to show that%
\begin{equation}
\gamma _{i,n+1}^{[s_{1}]}=\widehat{\gamma }_{i,n+1}^{[s_{1}]}=\alpha
_{i}\cdot \Theta _{n}^{[s_{1}]}=\alpha _{i}\cdot \lbrack (\gamma
_{1,n}^{[s_{1}]},\gamma _{2,n}^{[s_{1}]})\otimes (\gamma
_{1,n}^{[s_{2}]},\gamma _{2,n}^{[s_{2}]})]\text{,}  \label{14}
\end{equation}%
and $\gamma _{i,n+1}^{[s_{2}]}$ is constructed by letting all the
coefficients of 
\begin{equation*}
\widehat{\gamma }_{i,n+1}^{[s_{2}]}=\alpha _{i}\cdot \Theta
_{n}^{[s_{2}]}=\alpha _{i}\cdot \lbrack (\gamma _{1,n}^{[s_{1}]},\gamma
_{2,n}^{[s_{1}]})\otimes (1,1)]
\end{equation*}%
to be $1$ (Theorem \ref{Thm: 5}). Since $k_{1,2}=k_{2,2}$, we conclude that $%
\gamma _{1,n+1}^{[s_{2}]}=\gamma _{2,n+1}^{[s_{2}]}$. Combining (\ref{14})
with $\gamma _{1,n+1}^{[s_{2}]}=\gamma _{2,n+1}^{[s_{2}]}$ we can assert
that $\gamma _{1,n+1}^{[s_{1}]}=\gamma _{2,n+1}^{[s_{1}]}$, this proves the
claim.

\textbf{2} Since $\left\vert \alpha _{1}\right\vert =\left\vert \alpha
_{2}\right\vert $ and $\gamma _{1,n}^{[s_{j}]}=\gamma _{2,n}^{[s_{j}]}$ for $%
1\leq j\leq 2$. We have $\gamma _{i,n}^{[s_{1}]}=\left\vert \alpha
_{i}\right\vert \gamma _{1,n-1}^{[s_{1}]}\gamma _{1,n-1}^{[s_{2}]}=\alpha
\gamma _{1,n-1}^{[s_{1}]}\gamma _{1,n-1}^{[s_{2}]}$ and $\gamma
_{i,n}^{[s_{2}]}=\beta \gamma _{1,n-1}^{[s_{1}]}$. Thus, the nonlinear
equation $F=\{\gamma _{i,n}^{[s_{1}]},\gamma _{i,n}^{[s_{2}]}\}_{i=1}^{2}$
can be reduced to the simplified form 
\begin{equation}
\left\{ 
\begin{array}{c}
\gamma _{1,n}^{[s_{1}]}=\alpha \gamma _{1,n-1}^{[s_{1}]}\gamma
_{1,n-1}^{[s_{2}]}, \\ 
\gamma _{1,n}^{[s_{2}]}=\beta \gamma _{1,n-1}^{[s_{1}]}.%
\end{array}%
\right.  \label{29}
\end{equation}%
Let $w_{n}=(\ln \gamma _{1,n}^{[s_{1}]},\ln \gamma _{1,n}^{[s_{2}]})^{\prime
}$. We have $w_{n}=Aw_{n-1}+b$, where $b=(\ln \alpha ,\ln \beta )^{\prime }$%
. Iterate $w_{n}$ we have $w_{n}=A^{n-1}w_{1}+\sum_{i=0}^{n-2}A^{i}b$.
Observe that $w_{1}=b$, thus 
\begin{equation}
w_{n}=\sum_{i=0}^{n-1}A^{i}b=\ln \alpha \left( \sum_{i=0}^{n-1}A^{i}%
\widetilde{b}\right) ,  \label{20}
\end{equation}%
where $\widetilde{b}=(1,\frac{\ln \beta }{\ln \alpha })^{\prime }$.
Combining (\ref{17}) with the fact that $\gamma _{1,n}^{[s_{1}]}=\gamma
_{2,n}^{[s_{1}]}$ we assert that 
\begin{eqnarray}
h &=&\lim_{n\rightarrow \infty }\frac{\ln \sum_{i=1}^{2}\gamma
_{i,n}^{[s_{1}]}}{\left\vert E_{n}\right\vert }=\lim_{n\rightarrow \infty }%
\frac{\ln \gamma _{1,n}^{[s_{1}]}}{\left\vert E_{n}\right\vert }%
=\lim_{n\rightarrow \infty }\frac{w_{n}^{(1)}}{\left\vert E_{n}\right\vert }
\notag \\
&=&\lim_{n\rightarrow \infty }\frac{\left( \ln \alpha \right) \left(
\sum_{i=0}^{n-1}A^{i}\widetilde{b}\right) ^{(1)}}{\left\vert
E_{n}\right\vert }\text{.}  \label{5}
\end{eqnarray}%
Substituting $\left\vert E_{n}\right\vert =\left( \sum_{i=0}^{n}A^{i}\mathbf{%
1}\right) ^{(1)}$ into (\ref{5}) yields 
\begin{equation*}
h=\left( \ln \alpha \right) \lim_{n\rightarrow \infty }\frac{\left(
\sum_{i=0}^{n-1}A^{i}\widetilde{b}\right) ^{(1)}}{\left( \sum_{i=0}^{n}A^{i}%
\mathbf{1}\right) ^{(1)}}.
\end{equation*}%
Set $A=PDP^{-1}$, $P=(p_{ij})$, $P^{-1}=(q_{ij})$, and $D=diag(\lambda ,\bar{%
\lambda})$, we have%
\begin{eqnarray*}
&&\left( \sum_{i=0}^{n-1}A^{i}\widetilde{b}\right) ^{(1)}=\left[ P\left(
\sum_{i=0}^{n-1}D^{i}\right) P^{-1}\widetilde{b}\right] ^{(1)} \\
&=&\left[ P\left( \sum_{i=0}^{n-1}D^{i}\right) \left( q_{11}+q_{12}\frac{\ln
\beta }{\ln \alpha },q_{21}+q_{22}\frac{\ln \beta }{\ln \alpha }\right)
^{\prime }\right] ^{(1)} \\
&=&\sum_{i=0}^{n-1}\left[ \lambda ^{i}p_{11}\left( q_{11}+q_{12}\frac{\ln
\beta }{\ln \alpha }\right) +\bar{\lambda}^{i}p_{12}\left( q_{21}+q_{22}%
\frac{\ln \beta }{\ln \alpha }\right) \right] \\
&=&p_{11}\left( q_{11}+q_{12}\frac{\ln \beta }{\ln \alpha }\right) \frac{%
\lambda ^{n}-1}{\lambda -1}+p_{12}\left( q_{21}+q_{22}\frac{\ln \beta }{\ln
\alpha }\right) \frac{\bar{\lambda}^{n}-1}{\bar{\lambda}-1}\text{.}
\end{eqnarray*}%
It follows the same computation we have 
\begin{equation}
\left( \sum_{i=0}^{n-1}A^{i}\mathbf{1}\right) ^{(1)}=p_{11}\left(
q_{11}+q_{12}\right) \frac{\lambda ^{n}-1}{\lambda -1}+p_{12}\left(
q_{21}+q_{22}\right) \frac{\bar{\lambda}^{n}-1}{\bar{\lambda}-1}\text{.}
\label{28}
\end{equation}%
Thus 
\begin{equation*}
\lim_{n\rightarrow \infty }\frac{\left( \sum_{i=0}^{n-1}A^{i}\widetilde{b}%
\right) ^{(1)}}{\left( \sum_{i=0}^{n}A^{i}\mathbf{1}\right) ^{(1)}}=\frac{%
q_{11}+q_{12}\frac{\ln \beta }{\ln \alpha }}{\left( q_{11}+q_{12}\right)
\lambda }\text{.}
\end{equation*}%
Direct computation shows that 
\begin{equation*}
P=\left( 
\begin{array}{cc}
\frac{1+\sqrt{5}}{2} & \frac{1-\sqrt{5}}{2} \\ 
1 & 1%
\end{array}%
\right) \text{ and }P^{-1}=\frac{1}{\sqrt{5}}\left( 
\begin{array}{cc}
1 & \frac{-1+\sqrt{5}}{2} \\ 
-1 & \frac{1+\sqrt{5}}{2}%
\end{array}%
\right) \text{.}
\end{equation*}%
Thus $\lim_{n\rightarrow \infty }\frac{\left( \sum_{i=0}^{n-1}A^{i}\widehat{b%
}\right) ^{(1)}}{\left( \sum_{i=0}^{n}A^{i}\mathbf{1}\right) ^{(1)}}=\frac{1-%
\bar{\lambda}\frac{\ln \beta }{\ln \alpha }}{\lambda ^{2}}$. If $\alpha
=\beta $, we have $\frac{1-\bar{\lambda}\frac{\ln \beta }{\ln \alpha }}{%
\lambda ^{2}}=\frac{1}{\lambda }$. This completes the proof.
\end{proof}

\begin{example}
\begin{enumerate}
\item $h_{42}=\frac{1-\bar{\lambda}\frac{1}{2}}{\lambda ^{2}}\ln 4=\frac{1}{2%
}\ln 4=\ln 2$.

\item $h_{23}=\frac{1-\bar{\lambda}}{\lambda ^{2}}\ln 2=\frac{1}{\lambda }%
\ln 2\approx 0.428\,39$. Since $\deg (X)=\ln \lambda $ (cf. \cite%
{ban2017coloring}) we have $\gamma _{n}\approx $ $\left( 2^{\frac{1}{\lambda 
}}\right) ^{\lambda ^{n}}=2^{\lambda ^{n-1}}$, e.g., $\gamma _{9}\approx
\left( 2\right) ^{\lambda ^{8}}\approx 1.\,\allowbreak 386\,8\times 10^{14}$.

\item $h_{45}=\frac{1}{\lambda ^{2}}\ln 2\approx 0.264\,76$.
\end{enumerate}
\end{example}

\subsection{Dominating type}

Let $\gamma _{i,n}^{[s_{j}]}=\sum_{l=1}^{n_{ij}}f_{l,n-1}^{ij}$, where $%
f_{l,n-1}^{ij}$ denotes the $l$th item of $\gamma _{i,n}^{[s_{j}]}$. We say
that $\gamma _{i,n}^{[s_{j}]}$ has a \emph{dominate item} if there exists
integer $1\leq r\leq n_{ij}$ such that $f_{r,n}^{ij}\geq f_{l,n}^{ij}$ for
all $r\neq l$ and $n\geq 1$. We say $F=\{\gamma _{i,n}^{[s_{1}]},\gamma
_{i,n}^{[s_{2}]}\}_{i=1}^{2}$ is of the \emph{dominating type} ($F\in 
\mathbf{D}$\textbf{)} if each $\gamma _{i,n}^{[s_{j}]}$ has a dominate item
for all $1\leq i,j\leq 2$. If $F\in \mathbf{D}$, we assume that $%
f_{1,n-1}^{ij}$ is the corresponding dominate item for all $1\leq i,j\leq 2$%
. Thus, 
\begin{equation*}
\gamma
_{i,n}^{[s_{j}]}=\sum_{l=1}^{n_{ij}}f_{l,n-1}^{ij}=f_{1,n-1}^{ij}(1+%
\sum_{l=2}^{n_{ij}}\frac{f_{l,n-1}^{ij}}{f_{1,n-1}^{ij}})
\end{equation*}%
and $1\leq 1+\sum_{l=2}^{n_{ij}}\frac{f_{l,n-1}^{ij}}{f_{1,n-1}^{ij}}\leq 4$%
, where the number $4$ comes from the extreme case where $n_{ij}=4$ and $%
\frac{f_{l,n-1}^{ij}}{f_{1,n-1}^{ij}}\leq 1$ for $l=2,\ldots ,n_{ij}$. Let $%
w_{n}=(\ln \gamma _{1,n}^{[s_{1}]},\ln \gamma _{2,n}^{[s_{1}]},\ln \gamma
_{1,n}^{[s_{2}]},\ln \gamma _{2,n}^{[s_{2}]})^{\prime }$, it follows
immediately from (\ref{15}) and (\ref{16}) that 
\begin{equation}
w_{n}=Kw_{n-1}+b_{n-1},  \label{9}
\end{equation}%
for some $K\in \mathcal{M}_{4}$ and 
\begin{equation}
b_{n-1}=\left( 
\begin{array}{c}
\ln (1+\frac{\sum_{l=2}^{n_{11}}f_{l,n-1}^{11}}{f_{1,n-1}^{ij}}) \\ 
\ln (1+\frac{\sum_{l=2}^{n_{21}}f_{l,n-1}^{21}}{f_{1,n-1}^{ij}}) \\ 
\ln (1+\frac{\sum_{l=2}^{n_{12}}f_{l,n-1}^{12}}{f_{1,n-1}^{ij}}) \\ 
\ln (1+\frac{\sum_{l=2}^{n_{22}}f_{l,n-1}^{22}}{f_{1,n-1}^{ij}})%
\end{array}%
\right) \text{.}  \label{31}
\end{equation}%
Ban-Chang \cite{ban2018topological1} prove that if the symbol $\gamma
_{i}^{[s_{j}]}$ is \emph{essential}\footnote{%
We call the symbol $\gamma _{i}^{[s_{j}]}$ \emph{essential }if there exists $%
n\in \mathbb{N}$ such that $\gamma _{i,n}^{[s_{j}]}>1$, and \emph{%
inessential }otherwise. In \cite{ban2018topological}, the authors find a
finite checkable conditions to characterize whether $\gamma _{i}^{[s_{j}]}$
is essential or inessential.} for $1\leq i,j\leq 2$, then $\rho _{K}=\lambda 
$, where $\rho _{B}$ is the spectral radius of the matrix $B$. Let $v,w\in 
\mathbb{R}^{n}$ we say $v\geq w$ if $v_{i}\geq w_{i}$ for $1\leq i\leq n$.

\begin{proposition}
\label{Thm: 3}Suppose $\alpha _{1}>\alpha _{2}$ (or $\alpha _{2}>\alpha _{1}$%
), then $F\in \mathbf{D}$.
\end{proposition}

\begin{proof}
We only prove the case where $\alpha _{1}>\alpha _{2}$ and the other case is
similar. The proof is divided into small cases.

\textbf{1 }$\alpha _{1}=v_{1}$. In this case, there are eight possibilities
of $\alpha _{2}$, namely, $\alpha _{2}=v_{i}$ for $i=2,\ldots ,9$. If $%
\alpha _{2}=v_{2}$, the nonlinear recurrence equation $F_{12}=\{\gamma
_{i,n}^{[s_{1}]},\gamma _{i,n}^{[s_{2}]}\}_{i=1}^{2}$ is as follows%
\begin{equation*}
\left\{ 
\begin{array}{c}
\gamma _{1,n}^{[s_{1}]}=\gamma _{1,n-1}^{[s_{1}]}\gamma
_{1,n-1}^{[s_{2}]}+\gamma _{1,n-1}^{[s_{1}]}\gamma _{2,n-1}^{[s_{2}]}+\gamma
_{2,n-1}^{[s_{1}]}\gamma _{1,n-1}^{[s_{2}]}+\gamma _{2,n-1}^{[s_{1}]}\gamma
_{2,n-1}^{[s_{2}]}, \\ 
\gamma _{2,n}^{[s_{1}]}=\gamma _{1,n-1}^{[s_{1}]}\gamma
_{1,n-1}^{[s_{2}]}+\gamma _{2,n-1}^{[s_{1}]}\gamma _{1,n-1}^{[s_{2}]}, \\ 
\gamma _{1,n}^{[s_{2}]}=\gamma _{1,n-1}^{[s_{1}]}+\gamma _{2,n-1}^{[s_{1}]},
\\ 
\gamma _{2,n}^{[s_{2}]}=\gamma _{1,n-1}^{[s_{1}]}+\gamma _{2,n-1}^{[s_{1}]},
\\ 
\gamma _{1,1}^{[s_{1}]}=4,\gamma _{2,1}^{[s_{1}]}=2,\gamma
_{1,1}^{[s_{2}]}=1,\gamma _{2,1}^{[s_{2}]}=1.%
\end{array}%
\right.
\end{equation*}%
Since $\gamma _{1,n}^{[s_{2}]}\geq \gamma _{2,n}^{[s_{2}]}$ and $\gamma
_{1,n}^{[s_{1}]}\geq \gamma _{2,n}^{[s_{1}]}$ we deduce that $\gamma
_{1,n-1}^{[s_{1}]}\gamma _{1,n-1}^{[s_{2}]}$ (resp. $\gamma
_{1,n-1}^{[s_{1}]}$) is the dominate item for $\gamma _{1,n}^{[s_{1}]}$ and $%
\gamma _{2,n}^{[s_{1}]}$ (resp. $\gamma _{1,n}^{[s_{2}]}$ and $\gamma
_{2,n}^{[s_{2}]}$). Thus $F_{12}\in \mathbf{D}$. $F_{13},F_{14},F_{15}\in 
\mathbf{D}$ can be treated in the same manner. For $\alpha
_{2}=v_{6},v_{7},v_{8},v_{9}$, let $\alpha _{2}=v_{6}$, $F_{16}$ is of the
following form 
\begin{equation*}
\left\{ 
\begin{array}{c}
\gamma _{1,n}^{[s_{1}]}=\gamma _{1,n-1}^{[s_{1}]}\gamma
_{1,n-1}^{[s_{2}]}+\gamma _{1,n-1}^{[s_{1}]}\gamma _{2,n-1}^{[s_{2}]}+\gamma
_{2,n-1}^{[s_{1}]}\gamma _{1,n-1}^{[s_{2}]}+\gamma _{2,n-1}^{[s_{1}]}\gamma
_{2,n-1}^{[s_{2}]}, \\ 
\gamma _{2,n}^{[s_{1}]}=\gamma _{1,n-1}^{[s_{1}]}\gamma _{1,n-1}^{[s_{2}]},
\\ 
\gamma _{1,n}^{[s_{2}]}=\gamma _{1,n-1}^{[s_{1}]}+\gamma _{2,n-1}^{[s_{1}]},
\\ 
\gamma _{2,n}^{[s_{2}]}=\gamma _{1,n-1}^{[s_{1}]}, \\ 
\gamma _{1,1}^{[s_{1}]}=2,\gamma _{2,1}^{[s_{1}]}=1,\gamma
_{1,1}^{[s_{2}]}=2,\gamma _{2,1}^{[s_{2}]}=1.%
\end{array}%
\right.
\end{equation*}

Since $\gamma _{2,n}^{[s_{1}]}$ (resp. $\gamma _{2,n}^{[s_{2}]}$) have only
one item, it is the dominate item. It follows from the fact that $\gamma
_{1,n}^{[s_{1}]}\geq \gamma _{2,n}^{[s_{1}]}$ and $\gamma
_{1,n}^{[s_{2}]}\geq \gamma _{2,n}^{[s_{2}]}$, it is concluded that $\gamma
_{1,n-1}^{[s_{1}]}\gamma _{1,n-1}^{[s_{2}]}$ (resp. $\gamma
_{1,n-1}^{[s_{1}]}$) is still the dominate item of $\gamma _{1,n}^{[s_{1}]}$
(resp. $\gamma _{1,n}^{[s_{2}]}$). Thus, $F\in \mathbf{D}$. $%
F_{17},F_{18},F_{19}\in \mathbf{D}$ can be treated similarly.

\textbf{2 }$\alpha _{1}=v_{2},v_{3},v_{4},v_{5}$. Assume $\alpha _{1}=v_{2}$%
, since $\alpha _{1}>\alpha _{2}$, it suffices to check $\alpha _{2}=v_{6}$
and $v_{8}$. For $\alpha _{1}=v_{2}$, under the same argument as above we
conclude that $\gamma _{1,n-1}^{[s_{1}]}\gamma _{1,n-1}^{[s_{2}]}$ (resp. $%
\gamma _{1,n-1}^{[s_{1}]}$) is the dominate item for $\gamma
_{1,n}^{[s_{1}]} $ (resp. $\gamma _{2,n}^{[s_{1}]}$), thus $F_{26}\in 
\mathbf{D}$. The same reasoning applies to other cases. This completes the
proof.
\end{proof}

The entropy can be computed for $F\in \mathbf{D}$. Let us denote by $%
\mathcal{E}$ (resp. $\mathcal{I}$) the set of all essential (resp.
inessential) symbols of $\{\gamma _{i}^{[s_{j}]}\}_{i,j=1}^{2}$, we note
that $\{\gamma _{i}^{[s_{j}]}\}_{i,j=1}^{2}=\mathcal{E}\cup \mathcal{I}$.
The computation methods of $h$ is divided into two subcases, namely, $%
\mathcal{I}=\emptyset $ and $\mathcal{I}\neq \emptyset $. First, we take $%
F_{16}$ as an example to illustrate how to compute $h$ for this type, and $%
\mathcal{I}=\emptyset $ in this case.

\begin{theorem}
\label{Thm: 7}Let $X$ be a $G$-SFT in which the corresponding nonlinear
recurrence equation is $F_{16}$; that is, $T_{1}=T_{2}=\left( 
\begin{array}{cc}
1 & 1 \\ 
1 & 0%
\end{array}%
\right) $. Let $b_{n}$ be constructed as (\ref{31}), 
\begin{equation}
K=\left( 
\begin{array}{cccc}
1 & 0 & 1 & 0 \\ 
1 & 0 & 1 & 0 \\ 
1 & 0 & 0 & 0 \\ 
1 & 0 & 0 & 0%
\end{array}%
\right) \text{ and }Q=\left( 
\begin{array}{cccc}
\lambda & \overline{\lambda } & 0 & 0 \\ 
\lambda & \overline{\lambda } & 0 & 1 \\ 
1 & 1 & 0 & 0 \\ 
1 & 1 & 1 & 0%
\end{array}%
\right)  \label{23}
\end{equation}%
be such that $QDQ^{-1}=K$ and $D=diag(\lambda ,\bar{\lambda},0,0)$. Then 
\begin{equation*}
h_{16}=\frac{\left( \lambda -1\right) A_{\infty }}{\lambda ^{2}}\approx
0.236\,07A_{\infty },
\end{equation*}%
where 
\begin{equation}
A_{\infty }=\lim_{n\rightarrow \infty }\left( \widehat{w}_{1}^{(1)}+\lambda
^{-1}\widehat{b}_{1}^{(1)}+\cdots +\lambda ^{-n+1}\widehat{b}%
_{n}^{(1)}\right)  \label{32}
\end{equation}%
and $\widehat{b}_{n}=Q^{-1}b_{n}$. Moreover, the limit (\ref{32}) exists.
\end{theorem}

\begin{proof}
Note that 
\begin{equation*}
F_{16}=\left\{ 
\begin{array}{c}
\gamma _{1,n}^{[s_{1}]}=\gamma _{1,n-1}^{[s_{1}]}\gamma
_{1,n-1}^{[s_{2}]}+\gamma _{1,n-1}^{[s_{1}]}\gamma _{2,n-1}^{[s_{2}]}+\gamma
_{2,n-1}^{[s_{1}]}\gamma _{1,n-1}^{[s_{2}]}+\gamma _{2,n-1}^{[s_{1}]}\gamma
_{2,n-1}^{[s_{2}]}, \\ 
\gamma _{2,n}^{[s_{1}]}=\gamma _{1,n-1}^{[s_{1}]}\gamma _{1,n-1}^{[s_{2}]},
\\ 
\gamma _{1,n}^{[s_{2}]}=\gamma _{1,n-1}^{[s_{1}]}+\gamma _{2,n-1}^{[s_{1}]},
\\ 
\gamma _{2,n}^{[s_{2}]}=\gamma _{1,n-1}^{[s_{1}]}, \\ 
\gamma _{1,1}^{[s_{1}]}=4,\gamma _{2,1}^{[s_{1}]}=1,\gamma
_{1,1}^{[s_{2}]}=2,\gamma _{2,1}^{[s_{2}]}=1.%
\end{array}%
\right.
\end{equation*}%
It can be easily checked that $\gamma _{i,2}^{[s_{j}]}\geq 4$ for $1\leq
i,j\leq 2$. Thus, there is no inessential symbol and $\mathcal{I}=\emptyset $%
. Since $\gamma _{1,n}^{[s_{1}]}\geq \gamma _{2,n}^{[s_{1}]}$ and $\gamma
_{1,n}^{[s_{2}]}\geq \gamma _{2,n}^{[s_{2}]}$, $\gamma
_{1,n-1}^{[s_{1}]}\gamma _{1,n-1}^{[s_{2}]}$ (resp. $\gamma
_{1,n-1}^{[s_{1}]}$) is the dominate item of $\gamma _{1,n}^{[s_{1}]}$
(resp. $\gamma _{1,n}^{[s_{2}]}$), $F\in \mathbf{D}$. The above argument
indicates that $w_{n}=Kw_{n-1}+b_{n-1}$, and $K$ is indeed (\ref{23}). Along
the identical line of the proof in Theorem \ref{Thm: 2} we have $%
w_{n}=K^{n-1}w_{1}+\sum_{i=0}^{n-2}K^{i}b_{i}$, 
\begin{equation}
h=\lim_{n\rightarrow \infty }\frac{\ln \sum_{i=1}^{2}\gamma _{i,n}^{[s_{1}]}%
}{\left\vert E_{n}\right\vert }=\lim_{n\rightarrow \infty }\frac{\ln \gamma
_{1,n}^{[s_{1}]}}{\left\vert E_{n}\right\vert }=\lim_{n\rightarrow \infty }%
\frac{w_{n}^{(1)}}{\left\vert E_{n}\right\vert },  \label{13}
\end{equation}%
and 
\begin{eqnarray}
w_{n}^{(1)} &=&\left( K^{n-1}w_{1}+K^{n-2}b_{1}+\cdots +b_{n}\right) ^{(1)} 
\notag \\
&=&\left( QD^{n-1}Q^{-1}w_{1}+QD^{n-2}Q^{-1}b_{1}+\cdots
+QQ^{-1}b_{n}\right) ^{(1)}\text{.}  \label{18}
\end{eqnarray}%
Combining (\ref{18}) with direct computation yields 
\begin{eqnarray*}
w_{n}^{(1)} &=&\left( QD^{n-1}\widehat{w}_{1}+QD^{n-2}\widehat{b}_{1}+\cdots
+Q\widehat{b}_{n}\right) ^{(1)} \\
&=&\lambda ^{n-1}Q_{11}\left( \widehat{w}_{1}^{(1)}+\lambda ^{-1}\widehat{b}%
_{1}^{(1)}+\cdots +\lambda ^{-n+1}\widehat{b}_{n}^{(1)}\right) +O(\bar{%
\lambda}^{n}) \\
&=&\lambda ^{n}\left( \widehat{w}_{1}^{(1)}+\lambda ^{-1}\widehat{b}%
_{1}^{(1)}+\cdots +\lambda ^{-n+1}\widehat{b}_{n}^{(1)}\right) +O(\bar{%
\lambda}^{n})\text{.}
\end{eqnarray*}%
Combining (\ref{13}) we obtain that 
\begin{equation}
h=\lim_{n\rightarrow \infty }\frac{\lambda ^{n}\left( \widehat{w}%
_{1}^{(1)}+\lambda ^{-1}\widehat{b}_{1}^{(1)}+\cdots +\lambda ^{-n+1}%
\widehat{b}_{n}^{(1)}\right) }{\left( \sum_{i=0}^{n}A^{i}\mathbf{1}\right)
^{(1)}}\text{.}  \label{30}
\end{equation}%
Let $A_{\infty }=\lim_{n\rightarrow \infty }\left( \widehat{w}%
_{1}^{(1)}+\lambda ^{-1}\widehat{b}_{1}^{(1)}+\cdots +\lambda ^{-n+1}%
\widehat{b}_{n}^{(1)}\right) ,$ such limit exists due to the fact that $%
\widehat{b}_{n}^{(1)}$ is bounded for all $n$. Combining (\ref{30}) with (%
\ref{28}) yields 
\begin{eqnarray*}
h &=&\lim_{n\rightarrow \infty }\frac{\lambda ^{n}A_{\infty }}{\left(
\sum_{i=0}^{n}A^{i}\mathbf{1}\right) ^{(1)}} \\
&=&A_{\infty }\lim_{n\rightarrow \infty }\frac{\lambda ^{n}}{p_{11}\left(
q_{11}+q_{12}\right) \frac{\lambda ^{n}-1}{\lambda -1}+p_{12}\left(
q_{21}+q_{22}\right) \frac{\bar{\lambda}^{n}-1}{\bar{\lambda}-1}} \\
&=&\frac{A_{\infty }}{\frac{1}{\lambda -1}p_{11}\left( q_{11}+q_{12}\right) }
\\
&=&\frac{\left( \lambda -1\right) A_{\infty }}{\lambda ^{2}} \\
&\approx &0.236\,07A_{\infty }\approx 0.5011681177.
\end{eqnarray*}%
This completes the proof.
\end{proof}

Next, we use $F_{39}$ to illustrate the computation of $h$ for the case
where $\mathcal{I}\neq \emptyset $.

\begin{proposition}
\label{Prop: 2}$h_{39}=0$.
\end{proposition}

\begin{proof}
Note that%
\begin{equation}
F_{39}=\left\{ 
\begin{array}{c}
\gamma _{1,n}^{[s_{1}]}=\gamma _{1,n-1}^{[s_{1}]}\gamma
_{2,n-1}^{[s_{2}]}+\gamma _{2,n-1}^{[s_{1}]}\gamma _{2,n-1}^{[s_{2}]}, \\ 
\gamma _{2,n}^{[s_{1}]}=\gamma _{2,n-1}^{[s_{1}]}\gamma _{2,n-1}^{[s_{2}]},
\\ 
\gamma _{1,n}^{[s_{2}]}=\gamma _{1,n-1}^{[s_{1}]}+\gamma _{2,n-1}^{[s_{1}]},
\\ 
\gamma _{2,n}^{[s_{2}]}=\gamma _{2,n-1}^{[s_{1}]}, \\ 
\gamma _{1,1}^{[s_{1}]}=2,\gamma _{2,1}^{[s_{1}]}=1,\gamma
_{1,1}^{[s_{2}]}=2,\gamma _{2,1}^{[s_{2}]}=1.%
\end{array}%
\right.  \label{38}
\end{equation}%
It can be checked that $\gamma _{2,n}^{[s_{1}]}=\gamma _{2,n}^{[s_{2}]}=1$
for all $n\in \mathbb{N}$, and $F_{39}\in \mathbf{D}$. Thus, $\mathcal{I}%
=\{\gamma _{2}^{[s_{1}]},\gamma _{2}^{[s_{1}]}\}\neq \emptyset $. Under the
same argument in the beginning of this section, we construct 
\begin{equation*}
K=\left( 
\begin{array}{cccc}
1 & 0 & 0 & 1 \\ 
0 & 1 & 0 & 1 \\ 
1 & 0 & 0 & 0 \\ 
0 & 1 & 0 & 0%
\end{array}%
\right) \text{.}
\end{equation*}%
Since $w_{n}^{(2)}=w_{n}^{(4)}=0$, the formula (\ref{9}) can be reduced to
the following form. 
\begin{equation}
\widetilde{w}_{n}=\widetilde{K}\widetilde{w}_{n-1}+\widetilde{b}_{n-1}\text{,%
}  \label{33}
\end{equation}%
where $\widetilde{w}_{n}=(w_{n}^{(1)},w_{n}^{(3)})^{\prime }$, $\widetilde{b}%
_{n}=(\ln (1+\frac{\gamma _{2,n}^{[s_{1}]}}{\gamma _{1,n}^{[s_{1}]}}),\ln (1+%
\frac{\gamma _{2,n}^{[s_{1}]}}{\gamma _{1,n}^{[s_{1}]}}))$ and $\widetilde{K}%
=\left( 
\begin{array}{cc}
1 & 0 \\ 
1 & 0%
\end{array}%
\right) $ which is derived by deleting the $2$nd and $4$th columns and rows
of $K$. Note that if $\rho _{\widetilde{K}}>1$, the method used in Theorem %
\ref{Thm: 7} still works. The induction formula (\ref{33}) in fact provides
us the formula of $w_{n}^{(1)}$ and $w_{n}^{(3)}$. Indeed, since $\gamma
_{2,n}^{[s_{2}]}=1$, $w_{n}^{(1)}=\ln \gamma _{1,n}^{[s_{1}]}=\ln (\gamma
_{1,n-1}^{[s_{1}]}+\gamma _{2,n-1}^{[s_{1}]})=\ln \gamma
_{1,n}^{[s_{2}]}=w_{n}^{(3)}$, the recurrence equation (\ref{33}) is reduced
to the simple form 
\begin{eqnarray}
w_{n}^{(1)} &=&w_{n-1}^{(1)}+\ln (1+\frac{\gamma _{2,n-1}^{[s_{1}]}}{\gamma
_{1,n-1}^{[s_{1}]}})  \notag \\
&=&w_{1}^{(1)}+\ln (1+\frac{\gamma _{2,1}^{[s_{1}]}}{\gamma _{1,1}^{[s_{1}]}}%
)+\cdots +\ln (1+\frac{\gamma _{2,n-1}^{[s_{1}]}}{\gamma _{1,n-1}^{[s_{1}]}})
\notag \\
&=&\ln \gamma _{1,1}^{[s_{1}]}+\ln (1+\frac{\gamma _{2,1}^{[s_{1}]}}{\gamma
_{1,1}^{[s_{1}]}})+\cdots +\ln (1+\frac{\gamma _{2,n-1}^{[s_{1}]}}{\gamma
_{1,n-1}^{[s_{1}]}}).  \label{34}
\end{eqnarray}%
The equality (\ref{34}) actually demonstrates the inductive formula $%
w_{n}^{(1)}=\ln \gamma _{1,n}^{[s_{1}]}$. Indeed, 
\begin{eqnarray*}
&&\ln \gamma _{1,1}^{[s_{1}]}+\ln (1+\frac{\gamma _{2,1}^{[s_{1}]}}{\gamma
_{1,1}^{[s_{1}]}})+\cdots +\ln (1+\frac{\gamma _{2,n-1}^{[s_{1}]}}{\gamma
_{1,n-1}^{[s_{1}]}}) \\
&=&\ln \gamma _{1,1}^{[s_{1}]}(1+\frac{\gamma _{2,1}^{[s_{1}]}}{\gamma
_{1,1}^{[s_{1}]}})\cdots (1+\frac{\gamma _{2,n-1}^{[s_{1}]}}{\gamma
_{1,n-1}^{[s_{1}]}}) \\
&=&\ln \gamma _{1,1}^{[s_{1}]}(\frac{\gamma _{1,1}^{[s_{1}]}+\gamma
_{2,1}^{[s_{1}]}}{\gamma _{1,1}^{[s_{1}]}})\cdots (\frac{\gamma
_{1,n-1}^{[s_{1}]}+\gamma _{2,n-1}^{[s_{1}]}}{\gamma _{1,n-1}^{[s_{1}]}})%
\text{.} \\
&=&\ln \gamma _{1,1}^{[s_{1}]}(\frac{\gamma _{1,2}^{[s_{1}]}}{\gamma
_{1,1}^{[s_{1}]}})(\frac{\gamma _{1,3}^{[s_{1}]}}{\gamma _{1,2}^{[s_{1}]}}%
)\cdots (\frac{\gamma _{1,n}^{[s_{1}]}}{\gamma _{1,n-1}^{[s_{1}]}}) \\
&=&\ln \gamma _{1,n}^{[s_{1}]}.
\end{eqnarray*}%
The second equality comes from the recurrence formula (\ref{38}). For the
computation of $h$, we know that $\gamma _{1,n}^{[s_{1}]}=\gamma
_{1,n-1}^{[s_{1}]}+\gamma _{2,n-1}^{[s_{1}]}=\gamma _{1,n-1}^{[s_{1}]}+1$.
Thus $\gamma _{1,n}^{[s_{1}]}=n$ and $h=0$.
\end{proof}

The same reasoning applies to the cases $%
h_{62}=h_{64}=h_{72}=h_{58}=h_{59}=0 $.

\begin{corollary}
\label{Cor: 2}$h_{46}=h_{47}=h_{44}=h_{45}=\frac{1}{\lambda ^{2}}\ln 2$.
\end{corollary}

\begin{proof}
It follows from Theorem \ref{Thm: 2} we have $h_{44}=h_{45}$. It suffices to
prove that $h_{46}=h_{44}$, the case where $h_{47}=h_{44}$ is similar. Since 
\begin{equation*}
F_{46}=\left\{ 
\begin{array}{c}
\gamma _{1,n}^{[s_{1}]}=\gamma _{1,n-1}^{[s_{1}]}\gamma
_{1,n-1}^{[s_{2}]}+\gamma _{1,n-1}^{[s_{1}]}\gamma _{2,n-1}^{[s_{2}]}, \\ 
\gamma _{2,n}^{[s_{1}]}=\gamma _{1,n-1}^{[s_{1}]}\gamma _{1,n-1}^{[s_{2}]},
\\ 
\gamma _{1,n}^{[s_{2}]}=\gamma _{1,n-1}^{[s_{1}]}, \\ 
\gamma _{2,n}^{[s_{2}]}=\gamma _{1,n-1}^{[s_{1}]}, \\ 
\gamma _{1,1}^{[s_{1}]}=2,\gamma _{2,1}^{[s_{1}]}=1,\gamma
_{1,1}^{[s_{2}]}=1,\gamma _{2,1}^{[s_{2}]}=1.%
\end{array}%
\right.
\end{equation*}%
Thus, we have $\gamma _{1,n}^{[s_{2}]}=\gamma _{2,n}^{[s_{2}]}$, and it
follows from the fact that $F_{46}\in \mathbf{D}$, we reduce $F_{46}$ to
(note $h_{46}=\lim_{n\rightarrow \infty }\frac{\ln \left( \gamma
_{1,n}^{[s_{1}]}+\gamma _{2,n}^{[s_{1}]}\right) }{\left\vert
E_{n}\right\vert }=\lim_{n\rightarrow \infty }\frac{\ln \gamma
_{1,n}^{[s_{1}]}}{\left\vert E_{n}\right\vert }$) 
\begin{equation*}
\left\{ 
\begin{array}{c}
\gamma _{1,n}^{[s_{1}]}=2\gamma _{1,n-1}^{[s_{1}]}\gamma _{1,n-1}^{[s_{2}]},
\\ 
\gamma _{1,n}^{[s_{2}]}=\gamma _{1,n-1}^{[s_{1}]}, \\ 
\gamma _{1,1}^{[s_{1}]}=2,\gamma _{1,1}^{[s_{2}]}=1.%
\end{array}%
\right.
\end{equation*}%
The same argument as in the proof of Theorem \ref{Thm: 2} infers that $%
h_{46}=h_{44}=\frac{1}{\lambda ^{2}}\ln 2$. This completes the proof.
\end{proof}

\subsection{Oscillating type}

We call an $F=\{\gamma _{i,n}^{[s_{1}]},\gamma _{i,n}^{[s_{2}]}\}_{i=1}^{2}$
the \emph{oscillating type} ($F\in \mathbf{O}$\textbf{)} if there exist two
sequences $\{m_{n}^{1}\}$, $\{m_{n}^{2}\}$ of $\mathbb{N}$ with $%
\{m_{n}^{1}\}\cap \{m_{n}^{2}\}=\emptyset $ and $\{m_{n}^{1}\}\cup
\{m_{n}^{2}\}=\mathbb{N}$ such that $\gamma _{1,n}^{[s_{1}]}\geq \gamma
_{2,n}^{[s_{1}]}$ for $n\in \{m_{n}^{1}\}$ and $\gamma
_{1,n}^{[s_{1}]}<\gamma _{2,n}^{[s_{1}]}$ if $n\in \{m_{n}^{2}\}$. We say $%
F\in \mathbf{O}_{2}$ if the two sequences are odd and even numbers. For $%
F\in \mathbf{O}_{2}$, $h$ can be computed along the same line of Theorem \ref%
{Thm: 7}. The steps are listed as follows.

\begin{enumerate}
\item Expand $F=\{\gamma _{i,n}^{[s_{1}]},\gamma
_{i,n}^{[s_{2}]}\}_{i=1}^{2} $ to $(n-2)$-order, say $F^{(2)}$; that is,
expand each item of $\gamma _{i,n}^{[s_{j}]}$ to next level according to the
rule of $F$.

\item Since $\gamma _{1,n}^{[s_{1}]}\geq \gamma _{2,n}^{[s_{1}]}$ for $n$
being even and $\gamma _{1,n}^{[s_{1}]}<\gamma _{2,n}^{[s_{1}]}$ for $n$
being odd, one assures that $F^{(2)}\in \mathbf{D}$.

\item Construct $w_{n}=Kw_{n-2}+b_{n-2}$ as in the case of dominating type,
and note that $K\in \mathcal{M}_{4\times 4}$ with $\rho _{K}=\lambda ^{2}$
(cf. \cite{ban2018topological1}).

\item Iteration $w_{n}$ and compute the growth rate of $\lim_{n\rightarrow
\infty }\frac{w_{2n}}{\left\vert E_{2n}\right\vert }$. Since the limit $h$
exists (Theorem \ref{Thm: 4}), we have $h=\lim_{n\rightarrow \infty }\frac{%
w_{2n}}{\left\vert E_{2n}\right\vert }$.
\end{enumerate}

The following Proposition characterizes whether $F\in \mathbf{O}_{2}$.

\begin{proposition}
\label{Prop: 3}$F_{36},F_{56},F_{92},F_{94}\in \mathbf{O}_{2}$.
\end{proposition}

\begin{proof}
Note that $F_{36}\simeq F_{92}$ and $F_{56}\simeq F_{94}$. Thus we only need
to prove $F_{36}$ and $F_{56}$. Since the proofs of $F_{36}$ and $F_{56}$
are identical, it suffices to prove the case of $F_{36}$. $F_{36}$ is of the
following form 
\begin{equation}
F_{36}=\left\{ 
\begin{array}{c}
\gamma _{1,n}^{[s_{1}]}=\gamma _{1,n-1}^{[s_{1}]}\gamma
_{2,n-1}^{[s_{2}]}+\gamma _{2,n-1}^{[s_{1}]}\gamma _{2,n-1}^{[s_{2}]}, \\ 
\gamma _{2,n}^{[s_{1}]}=\gamma _{1,n-1}^{[s_{1}]}\gamma _{1,n-1}^{[s_{2}]},
\\ 
\gamma _{1,n}^{[s_{2}]}=\gamma _{1,n-1}^{[s_{1}]}+\gamma _{2,n-1}^{[s_{1}]},
\\ 
\gamma _{2,n}^{[s_{2}]}=\gamma _{1,n-1}^{[s_{1}]}, \\ 
\gamma _{1,1}^{[s_{1}]}=2,\gamma _{2,1}^{[s_{1}]}=1,\gamma
_{1,1}^{[s_{2}]}=2,\gamma _{2,1}^{[s_{2}]}=1.%
\end{array}%
\right.  \label{19}
\end{equation}%
Let $\tau _{n}=\frac{\gamma _{1,n}^{[s_{1}]}}{\gamma _{2,n}^{[s_{1}]}}$ and $%
\chi _{n}=\frac{\gamma _{1,n}^{[s_{2}]}}{\gamma _{2,n}^{[s_{2}]}}$, we have 
\begin{eqnarray*}
\tau _{n} &=&\frac{\gamma _{1,n-1}^{[s_{1}]}\gamma _{2,n-1}^{[s_{2}]}+\gamma
_{2,n-1}^{[s_{1}]}\gamma _{2,n-1}^{[s_{2}]}}{\gamma _{1,n-1}^{[s_{1}]}\gamma
_{1,n-1}^{[s_{2}]}}=\frac{1}{\chi _{n-1}}+\frac{1}{\tau _{n-1}\chi _{n-1}}%
\text{,} \\
\chi _{n} &=&\frac{\gamma _{1,n-1}^{[s_{1}]}+\gamma _{2,n-1}^{[s_{1}]}}{%
\gamma _{1,n-1}^{[s_{1}]}}=1+\frac{1}{\tau _{n-1}}\text{.}
\end{eqnarray*}%
The direct computation shows that $(\tau _{1},\chi _{1})=(\frac{1}{2},\frac{1%
}{2})$ and $(\tau _{2},\chi _{2})=(8,3)$. Note that if $\tau _{n}\leq \frac{1%
}{2}$ and $\chi _{n}\leq \frac{3}{2}$, then 
\begin{eqnarray*}
\tau _{n+1} &=&\frac{1}{\chi _{n}}+\frac{1}{\tau _{n}\chi _{n}}\geq \frac{2}{%
3}+\frac{4}{3}=2, \\
\chi _{n+1} &=&1+\frac{1}{\tau _{n-1}}\geq 1+2=3.
\end{eqnarray*}%
If $\tau _{n}\geq 2,$ $\chi _{n}\geq 3$, we have 
\begin{eqnarray*}
\tau _{n+1} &=&\frac{1}{\chi _{n}}+\frac{1}{\tau _{n}\chi _{n}}\leq \frac{1}{%
3}+\frac{1}{6}=\frac{1}{2}, \\
\chi _{n+1} &=&1+\frac{1}{\tau _{n-1}}\leq 1+\frac{1}{2}\leq \frac{3}{2}.
\end{eqnarray*}%
By induction we have $\tau _{n}\leq \frac{1}{2}$ ($\gamma
_{1,n}^{[s_{1}]}<\gamma _{2,n}^{[s_{1}]}$) and $\chi _{n}\leq \frac{3}{2}$
for $n$ being an odd number and $\tau _{n}\geq 2$ ($\gamma
_{1,n}^{[s_{1}]}>\gamma _{2,n}^{[s_{1}]}$)$,$ $\chi _{n}\geq 3$ for $n$
being even, i.e., $F\in \mathbf{O}_{2}$. This completes the proof.
\end{proof}

\section{Characterization}

\begin{theorem}
\label{Thm: 6}Let $\mathcal{A}=\{1,2\}$. Suppose $F=\{\gamma
_{i,n}^{[s_{1}]},\gamma _{i,n}^{[s_{2}]}\}_{i=1}^{2}$ is the nonlinear
recurrence equation of a $G$-SFT with $G=\left\langle S|R_{A}\right\rangle $,
where $A=$ $\left( 
\begin{array}{cc}
1 & 1 \\ 
1 & 0%
\end{array}%
\right) $. Then $F$ is either one of the following four types.

\begin{enumerate}
\item $F$ is of the zero entropy type.

\item $F$ is of the equal growth type,

\item $F$ is of the dominating type,

\item $F$ is of the oscillating type.
\end{enumerate}
\end{theorem}

\begin{proof}
Without loss of generality, we assume that $\left\vert \alpha
_{1}\right\vert \geq \left\vert \alpha _{2}\right\vert $. The proof is
divided into two subcases.

\textbf{1} $\left\vert \alpha _{1}\right\vert =4$. In this case, under the
same proof of Theorem \ref{Thm: 3} we see that $F\in \mathbf{D}$.

\textbf{2 }$\left\vert \alpha _{1}\right\vert =2$. Note that $\left( \alpha
_{1},\alpha _{2}\right) =(v_{k},v_{l})$ for $k,l=2,3$ and $\left( \alpha
_{1},\alpha _{2}\right) =(v_{4},v_{4})$ satisfy the assumptions of Theorem %
\ref{Thm: 2}, they belong to type $\mathbf{E}$. For other cases, since most
of them are type $\mathbf{D}$ under the routine check, we only pick those $F$
which do not belong to type $\mathbf{D}$, namely, $F_{36}$, $F_{56}$, $%
F_{38} $, $F_{57}$, $F_{72}$, $F_{92}$, $F_{94}$ and $F_{84}$. Since $%
F_{36}\simeq F_{92}$, $F_{56}\simeq F_{94}$, $F_{38}\simeq F_{72}$ and $%
F_{57}\simeq F_{84}$. It suffices to check $F_{36},F_{56},F_{38}$ and $%
F_{57} $. Proposition \ref{Prop: 3} indicates that $F_{36},F_{56}\in \mathbf{%
O}_{2}$. Thus we only need to discuss the cases of $F_{38}$ and $F_{57}$.
Actually, we prove $F_{38}\in \mathbf{O}\backslash \mathbf{O}_{2}$ ($F_{57}$
is similar). $F_{38}$ is of the form 
\begin{equation*}
\left\{ 
\begin{array}{c}
\gamma _{1,n}^{[s_{1}]}=\gamma _{1,n-1}^{[s_{1}]}\gamma
_{2,n-1}^{[s_{2}]}+\gamma _{2,n-1}^{[s_{1}]}\gamma _{2,n-1}^{[s_{2}]}, \\ 
\gamma _{2,n}^{[s_{1}]}=\gamma _{2,n-1}^{[s_{1}]}\gamma _{1,n-1}^{[s_{2}]},
\\ 
\gamma _{1,n}^{[s_{2}]}=\gamma _{1,n-1}^{[s_{1}]}+\gamma _{2,n-1}^{[s_{1}]},
\\ 
\gamma _{2,n}^{[s_{2}]}=\gamma _{2,n-1}^{[s_{1}]}, \\ 
\gamma _{1,1}^{[s_{1}]}=2,\gamma _{2,1}^{[s_{1}]}=1,\gamma
_{1,1}^{[s_{2}]}=2,\gamma _{2,1}^{[s_{2}]}=1.%
\end{array}%
\right.
\end{equation*}%
Let $\tau _{n}=$ $\frac{\gamma _{1,n}^{[s_{1}]}}{\gamma _{2,n}^{[s_{1}]}}$
and $\chi _{n}=\frac{\gamma _{1,n}^{[s_{2}]}}{\gamma _{2,n}^{[s_{2}]}}$,
then 
\begin{equation*}
\tau _{n}=\frac{\tau _{n-1+1}}{\chi _{n-1}}=\frac{\chi _{n}}{\chi _{n-1}}%
\text{ and }\chi _{n}=\tau _{n-1}+1\text{.}
\end{equation*}%
Direct examination shows that $(\tau _{1},\chi _{1})=(2,2)$, $(\tau
_{2},\chi _{2})=(\frac{3}{2},3),$ $(\tau _{3},\chi _{3})=(\frac{5}{6},\frac{5%
}{2}),(\tau _{4},\chi _{4})=(\frac{11}{15},\frac{11}{6}),(\tau _{5},\chi
_{5})=(\frac{52}{55},\frac{26}{15})\ldots $, thus $F_{38}\notin \mathbf{O}%
_{2}$.

\textbf{3} $\left\vert \alpha _{1}\right\vert =1$. Proposition \ref{Prop: 1}
indicates that all these cases belong to type $\mathbf{Z}$. This completes
the proof.
\end{proof}

The following table indicates all types of $F_{ij}$ for $1\leq i,j\leq 9$. 
\begin{equation*}
\begin{tabular}{|l|l|l|l|l|l|l|l|l|l|}
\hline
$\alpha _{2}\backslash \alpha _{1}$ & $v_{1}$ & $v_{2}$ & $v_{3}$ & $v_{4}$
& $v_{5}$ & $v_{6}$ & $v_{7}$ & $v_{8}$ & $v_{9}$ \\ \hline
$v_{1}$ & $\mathbf{E}$ & $\mathbf{D}$ & $\mathbf{D}$ & $\mathbf{D}$ & $%
\mathbf{D}$ & $\mathbf{D}$ & $\mathbf{D}$ & $\mathbf{D}$ & $\mathbf{D}$ \\ 
\hline
$v_{2}$ & $\mathbf{D}$ & $\mathbf{E}$ & $\mathbf{E}$ & $\mathbf{D}$ & $%
\mathbf{D}$ & $\mathbf{D}$ & $\mathbf{O}$ & $\mathbf{D}$ & $\mathbf{O}_{2}$
\\ \hline
$v_{3}$ & $\mathbf{D}$ & $\mathbf{E}$ & $\mathbf{E}$ & $\mathbf{D}$ & $%
\mathbf{D}$ & $\mathbf{D}$ & $\mathbf{D}$ & $\mathbf{D}$ & $\mathbf{D}$ \\ 
\hline
$v_{4}$ & $\mathbf{D}$ & $\mathbf{D}$ & $\mathbf{D}$ & $\mathbf{E}$ & $%
\mathbf{E}$ & $\mathbf{D}$ & $\mathbf{D}$ & $\mathbf{O}$ & $\mathbf{O}_{2}$
\\ \hline
$v_{5}$ & $\mathbf{D}$ & $\mathbf{D}$ & $\mathbf{D}$ & $\mathbf{E}$ & $%
\mathbf{E}$ & $\mathbf{D}$ & $\mathbf{D}$ & $\mathbf{D}$ & $\mathbf{D}$ \\ 
\hline
$v_{6}$ & $\mathbf{D}$ & $\mathbf{D}$ & $\mathbf{O}_{2}$ & $\mathbf{D}$ & $%
\mathbf{O}_{2}$ & $\mathbf{Z}$ & $\mathbf{Z}$ & $\mathbf{Z}$ & $\mathbf{Z}$
\\ \hline
$v_{7}$ & $\mathbf{D}$ & $\mathbf{D}$ & $\mathbf{D}$ & $\mathbf{D}$ & $%
\mathbf{O}$ & $\mathbf{Z}$ & $\mathbf{Z}$ & $\mathbf{Z}$ & $\mathbf{Z}$ \\ 
\hline
$v_{8}$ & $\mathbf{D}$ & $\mathbf{D}$ & $\mathbf{O}$ & $\mathbf{D}$ & $%
\mathbf{D}$ & $\mathbf{Z}$ & $\mathbf{Z}$ & $\mathbf{Z}$ & $\mathbf{Z}$ \\ 
\hline
$v_{9}$ & $\mathbf{D}$ & $\mathbf{D}$ & $\mathbf{D}$ & $\mathbf{D}$ & $%
\mathbf{D}$ & $\mathbf{Z}$ & $\mathbf{Z}$ & $\mathbf{Z}$ & $\mathbf{Z}$ \\ 
\hline
\end{tabular}%
\end{equation*}

\subsection{Numerical results}

The numerical result of $h$ is presented. We give some explanations as
follows.

\begin{enumerate}
\item We note that for each $F_{kl}$, there exists a unique $F_{pq}$ such
that $F_{kl}\simeq F_{pq}$, which gives $h_{kl}=h_{pq}$, e.g., $%
h_{48}=h_{75} $, $h_{14}=h_{51}$ etc.

\item $h_{kl}=0$ for $k,l\in \{6,7,8,9\}$ (Proposition \ref{Prop: 1}).

\item $h_{11}=\ln 2$ and $h_{kl}=\frac{\ln 2}{\lambda }$ if $k,l\in \{2,3\}$
and $h_{kl}=\frac{\ln 2}{\lambda ^{2}}$ if $k,l\in \{4,5\}$ (Theorem \ref%
{Thm: 2}).

\item $h_{44}=h_{45}=h_{46}=h_{47}$ (Corollary \ref{Cor: 2}).

\item $h_{39}=h_{62}=h_{64}=h_{72}=h_{58}=h_{59}=0$ (Proposition \ref{Prop:
2}).
\end{enumerate}

\begin{equation*}
\begin{tabular}{|l|l|l|l|l|}
\hline
$\alpha _{2}\backslash \alpha _{1}$ & $v_{1}$ & $v_{2}$ & $v_{3}$ & $v_{4}$
\\ \hline
$v_{1}$ & 0.6924441915 & 0.5827398718 & 0.5827398718 & 0.4446025684 \\ \hline
$v_{2}$ & 0.5827398718 & 0.4282225063 & 0.4282225063 & 0.3480809809 \\ \hline
$v_{3}$ & 0.5827398718 & 0.4282225063 & 0.4282225063 & 0.3785719508 \\ \hline
$v_{4}$ & 0.5473654583 & 0.4046074815 & 0.2920492775 & 0.2648611178 \\ \hline
$v_{5}$ & 0.4446025684 & 0.3785719508 & 0.3480809809 & 0.2648611178 \\ \hline
$v_{6}$ & 0.5011681177 & 0.3384608728 & 0.2437451279 & 0.2647497426 \\ \hline
$v_{7}$ & 0.4808946783 & 0.3062336239 & 0.2093165951 & 0.2647497426 \\ \hline
$v_{8}$ & 0.3742043181 & 0.2747387680 & 0.1904155180 & 0.2009045358 \\ \hline
$v_{9}$ & 0.3529894045 & 0.2396006045 & 0 & 0.1959187210 \\ \hline
\end{tabular}%
\end{equation*}%
\begin{equation*}
\begin{tabular}{|l|l|l|l|l|}
\hline
$\alpha _{2}\backslash \alpha _{1}$ & $v_{5}$ & $v_{6}$ & $v_{7}$ & $v_{8}$
\\ \hline
$v_{1}$ & 0.5473654583 & 0.3529894045 & 0.3742043181 & 0.4808946783 \\ \hline
$v_{2}$ & 0.2920492775 & 0 & 0.1904155180 & 0.2093165951 \\ \hline
$v_{3}$ & 0.4046074815 & 0.2396006045 & 0.2747387680 & 0.3062336239 \\ \hline
$v_{4}$ & 0.2648611178 & 0 & 0 & 0.1486957616 \\ \hline
$v_{5}$ & 0.2648611178 & 0.1959187210 & 0.2009045358 & 0.2648611178 \\ \hline
$v_{6}$ & 0.1312568310 & 0 & 0 & 0 \\ \hline
$v_{7}$ & 0.1486957616 & 0 & 0 & 0 \\ \hline
$v_{8}$ & 0 & 0 & 0 & 0 \\ \hline
$v_{9}$ & 0 & 0 & 0 & 0 \\ \hline
\end{tabular}%
\end{equation*}%
\begin{equation*}
\begin{tabular}{|l|l|}
\hline
$\alpha _{2}\backslash \alpha _{1}$ & $v_{9}$ \\ \hline
$v_{1}$ & 0.5011681177 \\ \hline
$v_{2}$ & 0.2437451279 \\ \hline
$v_{3}$ & 0.3384608728 \\ \hline
$v_{4}$ & 0.1312568310 \\ \hline
$v_{5}$ & 0.2648611178 \\ \hline
$v_{6}$ & 0 \\ \hline
$v_{7}$ & 0 \\ \hline
$v_{8}$ & 0 \\ \hline
$v_{9}$ & 0 \\ \hline
\end{tabular}%
\end{equation*}

\section{Conclusion and open problems}

We list the results of this paper as follows.

\begin{enumerate}
\item The existence of the entropy (\ref{21}) for a $G$-shift is illustrated
in Theorem \ref{Thm: 4}.

\item The nonlinear recurrence equation which describes the growth behavior
of the admissible patterns $\gamma _{n}(X)$ in a $G$-SFT (or $G$-vertex
shift) is established in Section 2.

\item The $\mathbf{Z,E},\mathbf{D}$ and $\mathbf{O}$ ($\mathbf{O}_{2}$)
types of nonlinear recurrence equations are introduced. The algorithms of the
entropy computations for these types are also presented (cf. Section 3).

\item The characterization of the nonlinear recurrence equations of $G$-SFTs
with two symbols is presented (Theorem \ref{Thm: 6}).
\end{enumerate}

We emphasize that the computation method of $h$ can be easily extended to
the case of more symbols. However, the general entropy formula for arbitrary 
$G$-SFTs is far from being solved. We list some problems in the further
study.

\begin{problem}
Can we give the characterization for $G$-SFTs over symbol set $\mathcal{A}$
with $\left\vert \mathcal{A}\right\vert >2$?
\end{problem}

\begin{problem}
Let $H=\left\langle S|R_{B}\right\rangle $ with $S=\{s_{1},\ldots ,s_{d}\}$
and $B$ be an arbitrary $d$-dimensional $\{0,1\}$-matrix, can we develop the
entropy theory for $H$-SFTs?
\end{problem}

\begin{problem}
Can we extend the methods of $G$-SFTs to $F_{d}$-SFTs? More precisely, can
we establish the entropy formula for $F_{d}$-SFTs?
\end{problem}

% bibliography ---------------------------------------------------
\bibliographystyle{amsplain}
\bibliography{ban}

% -------------------------------------------------------------

\end{document}